\documentclass[a4paper,11pt,reqno]{amsart}
\pdfoutput=1

\usepackage{amsmath,amssymb,amsthm,mathrsfs}
\usepackage{mathtools}
\usepackage[all,cmtip]{xy}
\usepackage{bm}
\usepackage[a4paper,margin=1in]{geometry}
\usepackage{lmodern}
\usepackage[T1]{fontenc}

\usepackage{comment}

\usepackage{microtype}

\usepackage[UKenglish]{isodate}

\usepackage[shortlabels]{enumitem}

\swapnumbers

\theoremstyle{plain}

\newtheorem{theorem}[subsection]{Theorem}
\newtheorem{proposition}[subsection]{Proposition}
\newtheorem{lemma}[subsection]{Lemma}
\newtheorem{corollary}[subsection]{Corollary}
\newtheorem{conjecture}[subsection]{Conjecture}

\theoremstyle{definition} 
\newtheorem{definition}[subsection]{Definition}
\newtheorem{example}[subsection]{Example}

\theoremstyle{remark}
\newtheorem{remark}[subsection]{Remark}
\newtheorem{observation}[subsection]{Observation}
\newtheorem{recall}[subsection]{Recall}
\newtheorem{perspective}[subsection]{Perspective}

\makeatletter
\let\c@subsection\c@equation
\makeatother
\numberwithin{equation}{section}

\usepackage[hidelinks]{hyperref}

\newcommand{\lra}{\longrightarrow}
\newcommand{\Hom}{\mathrm{Hom}}
\newcommand{\bHom}{\mathbf{Hom}}
\newcommand{\Fun}{\mathbf{Fun}}
\newcommand{\Set}{\mathbf{Set}}
\newcommand{\Cat}{\mathbf{Cat}}
\newcommand{\sSet}{\mathbf{sSet}}
\newcommand{\ssSet}{\mathbf{ssSet}}
\newcommand{\ho}{\operatorname{ho}}
\newcommand{\ob}{\operatorname{ob}}
\newcommand{\cd}[2][]{\vcenter{\hbox{\xymatrix#1{#2}}}}
\newcommand{\dtwocell}[3][0.5]{\ar@{}[#2] \ar@{=>}?(#1)+/u  0.2cm/;?(#1)+/d 0.2cm/^{#3}}
\newcommand{\fatpullbackcorner}[1][dr]{\save*!/#1-1.75pc/#1:(-1,1)@^{|-}\restore}
\newcommand{\fatleftpullbackcorner}[1][dl]{\save*!/#1-1.75pc/#1:(-1,1)@^{|-}\restore}
\newcommand{\fatpushoutcorner}[1][dr]{\save*!/#1+1.75pc/#1:(1,-1)@^{|-}\restore}
\newcommand{\Cyl}{\mathbf{Cyl}}

\newcommand{\pullbackcorner}[1][dr]{\save*!/#1-1.2pc/#1:(-1,1)@^{|-}\restore}
\newcommand{\leftpullbackcorner}[1][dl]{\save*!/#1-1.2pc/#1:(-1,1)@^{|-}\restore}

\newcommand{\hdash}{\rotatebox[origin=c]{90}{$\vdash$}}

\title{Joyal's cylinder conjecture}
\author{Alexander Campbell} 
\address{Centre of Australian Category Theory \\ Macquarie University \\ NSW 2109 \\ Australia}
\urladdr{http://web.science.mq.edu.au/~alexc/}

\subjclass[2010]{18G30, 18G55, 55U10, 55U35}
\date{7 November 2019}

\begin{document}

\begin{abstract}
For each pair of simplicial sets $A$ and $B$, the category $\Cyl(A,B)$ of cylinders (also called correspondences) from $A$ to $B$ admits a model structure induced from Joyal's model structure for quasi-categories. In this paper, we prove Joyal's conjecture that a cylinder $X \in \Cyl(A,B)$ is fibrant if and only if the canonical morphism $X \lra A \star B$ is an inner fibration, and that a morphism between fibrant cylinders in $\Cyl(A,B)$ is a fibration if and only if it is an inner fibration. We use this result to give a new proof of a characterisation of covariant equivalences due to Lurie, which avoids the use of the straightening theorem. 
In an appendix, we introduce a new family of model structures on the slice categories $\sSet/B$, whose fibrant objects are the inner fibrations with codomain $B$, which we use to prove some new results about inner anodyne extensions and inner fibrations.
\end{abstract}

\maketitle

\tableofcontents

\section{Introduction} \label{secintro}
Recall that the \emph{collage}\footnote{This construction is due to B\'enabou \cite{ben2dimlim,bendist}; its name is due to Street and Walters (see \cite{MR2048315}).} of a profunctor $M \colon A \lra B$ (i.e.\ a functor $M \colon A^\mathrm{op}\times B\lra \Set$) is the category $C(M)$ whose set of objects is the disjoint union $\ob C(M) = \ob A + \ob B$, whose hom-sets are given by
\begin{equation*}
C(M)(x,y) = \begin{cases}
A(x,y) &\text{if $x,y \in A$} \\
B(x,y) &\text{if $x,y \in B$} \\
M(x,y) &\text{if $x \in A$, $y \in B$} \\
\emptyset &\text{if $x \in B$, $y \in A$,}
\end{cases}
\end{equation*}
and whose identities and composition are defined in the evident way by those of the categories $A$ and $B$, and by the action of $M$ on morphisms.  
There is a unique functor $C(M) \lra \mathbf{2} = \{0 < 1\}$ whose fibres above $0$ and $1$ are $A$ and $B$ respectively. 
B\'enabou observed that the collage construction defines an equivalence 
between the category of profunctors (between arbitrary categories) and the slice category $\Cat/\mathbf{2}$ (see \cite{stpow}).

In quasi-category theory, the category of \emph{cylinders} (or \emph{correspondences}) is defined to be the slice category $\sSet/\Delta[1]$.  
By analogy with the previous paragraph, a cylinder $p \colon X \lra \Delta[1]$ may be thought of as a model for the collage of a quasi-categorical profunctor from $\partial_0X := p^{-1}(0)$ to $\partial_1X := p^{-1}(1)$. 
(See \cite[Chapter 7]{joyalbarcelona}, \cite[\S14]{joyalnotes}, \cite[\S2.3.1]{MR2522659}, and \cite{stevbifib} for further details and intuition concerning cylinders/correspondences.)

For each pair of simplicial sets $A$ and $B$, the category $\Cyl(A,B)$ of \emph{cylinders from $A$ to $B$}, or \emph{$(A,B)$-cylinders}, is defined to be the fibre of the functor 
\begin{equation*}
(\partial_0,\partial_1) \colon \sSet/\Delta[1] \lra \sSet \times \sSet
\end{equation*}
 over the object $(A,B)$. 
Thus an object of $\Cyl(A,B)$ is a simplicial set $X$ (the \emph{underlying simplicial set} of the cylinder) equipped with a map $X \lra \Delta[1]$ whose fibres above $0$ and $1$ are $A$ and $B$ respectively, as displayed below.
\begin{equation*}
\cd{
A \ar[r] \ar[d] \pullbackcorner & X \ar[d]  & \ar[l] B \leftpullbackcorner \ar[d] \\
\{0\} \ar[r] & \Delta[1] & \{1\} \ar[l]
}
\end{equation*}
 Note that the initial and terminal objects of $\Cyl(A,B)$ are the disjoint union $A + B$ and join $A \star B$ respectively, equipped with the manifest structure maps. Hence, for each cylinder $X \in \Cyl(A,B)$, there exist canonical morphisms $A + B \lra X$ and $X \lra A \star B$.

In 
\cite[\S14.6]{joyalnotes}, Joyal described a model structure on $\Cyl(A,B)$ -- which we call the \emph{Joyal model structure} on $\Cyl(A,B)$ -- created\footnote{A model structure on a category $\mathcal{A}$ is said to be \emph{created} by a functor $F \colon \mathcal{A} \lra \mathcal{C}$ from a model structure on $\mathcal{C}$ if a morphism $f$ of $\mathcal{A}$ is a cofibration, weak equivalence, or fibration in the model structure on $\mathcal{A}$ if and only if the morphism $Ff$ is a cofibration, weak equivalence, or fibration respectively in the model structure on $\mathcal{C}$.}
 by the forgetful functor $\Cyl(A,B) \lra \sSet$ from the Joyal model structure for quasi-categories on $\sSet$, 
about which he made the following conjecture. (Note that, by definition, an object $X \in \Cyl(A,B)$ is fibrant in this model structure if and only if the canonical morphism $X \lra A \star B$ is a fibration in the Joyal model structure for quasi-categories on $\sSet$.)

\begin{conjecture}[Joyal] \label{joyalconj}
A cylinder $X \in \Cyl(A,B)$ is fibrant if and only if the canonical morphism $X \lra A \star B$ is an inner fibration, and a morphism between fibrant cylinders in $\Cyl(A,B)$ is a fibration if and only if it is an inner fibration.
\end{conjecture}

The main goal of this paper is to prove this conjecture (see Theorem \ref{conjthm}; see \S\ref{proofoutline} below for an outline of our proof). Note that the special case of this conjecture  in which $A$ and $B$ are quasi-categories is easily proven (see \cite[Lemmas 3.7 and 3.8]{stevbifib} or Lemma \ref{easycase} below); we prove it for \emph{every} pair of simplicial sets $A$ and $B$.

\begin{remark}
For many years, it was an open question whether every monic bijective-on-$0$-simplices weak categorical equivalence is inner anodyne (see \cite[\S2.10]{joyalnotes}). Were this so, the general case of Joyal's conjecture would be as easy to prove as the special case in which $A$ and $B$ are quasi-categories (cf.\ Remark \ref{ifonly}). However,  the author recently proved \cite{contrasample} that this is not so; hence a different argument is required to prove Joyal's conjecture.
\end{remark}

Our proof of Joyal's conjecture is contained in \S\S\ref{secmodstrcyl}--\ref{joyalconjsec} of this paper. (The contents of each section may be gleaned from its opening paragraph.) In the final section \S\ref{luriehalfsec}, we use this result to give a new proof of Lurie's characterisation of covariant equivalences (see Theorem \ref{luriecor}), which avoids the use of the straightening theorem \cite[Theorem 2.2.1.2]{MR2522659}. 
In an appendix (Appendix \ref{secparajoyal}), we introduce a new family of model structures on the slice categories $\sSet/B$ (which we call the \emph{parametrised Joyal model structures})  whose fibrant objects are the inner fibrations with codomain $B$, using which we prove some new results about inner anodyne extensions and inner fibrations.

\subsection{Outline of proof} \label{proofoutline}
Our proof of Joyal's conjecture may be outlined as follows.
\begin{enumerate}[leftmargin=*]
\item \label{stepone} We construct (in Theorem \ref{ambicylmod}) a model structure on $\Cyl(A,B)$, which we call the \emph{ambivariant model structure}, which has the same cofibrations as the Joyal model structure (i.e.\ the monomorphisms), but whose (fibrations between) fibrant objects are precisely as described in Joyal's conjecture. 
\item  \label{steptwo} We observe that Joyal's conjecture is therefore equivalent to the statement that, on the category $\Cyl(A,B)$, the Joyal model structure and the ambivariant model structure coincide. In particular, we know that these two model structures do coincide if $A$ and $B$ are quasi-categories (see Corollary \ref{easycor}).
\item \label{stepthree} Since every fibration in the Joyal model structure is in particular an inner fibration, it follows that every ambivariant equivalence in $\Cyl(A,B)$ is a weak categorical equivalence. It remains to prove the converse. 
\item \label{stepfour} For each pair of weak categorical equivalences $u \colon A \lra A'$ and $v \colon B \lra B'$ in $\sSet$, we prove that the pushforward functor (see \S\ref{ppadjsec})
$$(u,v)_! \colon \Cyl(A,B) \lra \Cyl(A',B') $$
\begin{enumerate}[(a)]
\item preserves weak categorical equivalences (see Proposition \ref{pfqadj}), and
\item reflects ambivariant equivalences (see Theorem \ref{pfqequiv}).
\end{enumerate}
\item \label{stepfive} It then remains to argue as follows (see Theorem \ref{mainthm}). Let $u \colon A \lra A'$ and $v \colon B \lra B'$ be weak categorical equivalences in $\sSet$ such that $A'$ and $B'$ are quasi-categories.  
For any morphism $f$ in $\Cyl(A,B)$, $f$ is a weak categorical equivalence $\implies$ $(u,v)_!(f)$ is a weak categorical equivalence (by (\ref{stepfour})(a)) $\implies$ $(u,v)_!(f)$ is an ambivariant equivalence (by (\ref{steptwo}), since $A'$ and $B'$ are quasi-categories) $\implies$ $f$ is an ambivariant equivalence (by (\ref{stepfour})(b)). This completes the proof.
\end{enumerate}

\section{Model structures for cylinders} \label{secmodstrcyl}
The goal of this section is to carry out step (\ref{stepone}) of the proof of Joyal's conjecture outlined in \S\ref{proofoutline}. For each pair of simplicial sets $A$ and $B$, we construct both the \emph{Joyal model structure} (Theorem \ref{joyalcylmod}) and the \emph{ambivariant model structure} (Theorem \ref{ambicylmod})  on the category $\Cyl(A,B)$ described in \S\ref{secintro}. (Note that the construction of the latter model structure involves the parametrised Joyal model structures introduced in Appendix \ref{secparajoyal}.) We shall construct both of these model structures by the following general technique. 

\subsection{Restricting model structures}
We say that a model structure on a category $\mathcal{C}$ \emph{restricts} to a model structure on a full subcategory $\mathcal{A}$ of $\mathcal{C}$ if the full inclusion $\mathcal{A} \lra \mathcal{C}$ creates a model structure on $\mathcal{A}$  from the model structure on $\mathcal{C}$, that is, if the classes consisting of those morphisms in $\mathcal{A}$ which are cofibrations, weak equivalences, and fibrations respectively in the model structure on $\mathcal{C}$ form a model structure on $\mathcal{A}$.

\begin{proposition} \label{restprop}
Let $\mathcal{C}$ be a cofibrantly generated model category, and let $\mathcal{A}$ be a full subcategory of $\mathcal{C}$ which is both reflective and coreflective via an adjoint triple $L \dashv I \dashv R$.
Then the model structure on $\mathcal{C}$ restricts to a model structure on $\mathcal{A}$ if and only if the adjunction
\begin{equation*}
\xymatrix{
\mathcal{C} \ar@<-1.5ex>[rr]^-{\hdash}_-{IR} && \ar@<-1.5ex>[ll]_-{IL} \mathcal{C}}
\end{equation*}
is a Quillen adjunction.
\end{proposition}
\begin{proof}
If the model structure on $\mathcal{C}$ restricts to one on $\mathcal{A}$, then the adjunctions $L \dashv I$ and $I\dashv R$ are Quillen adjunctions, and hence so is their composite $IL \dashv IR$.

Conversely, suppose that the adjunction $IL \dashv IR$ is a Quillen adjunction. Note that the category $\mathcal{A}$ is complete and cocomplete, since it is a (co)reflective subcategory of the complete and cocomplete category $\mathcal{C}$. By \cite[Theorem 2.3]{MR3937590}, the category $\mathcal{A}$ admits a model structure in which a morphism $f$ is a fibration (resp.\ weak equivalence) if and only if $If$ is a fibration (resp.\ weak equivalence) in the model category $\mathcal{C}$; furthermore, the adjunctions $L \dashv I$ and $I \dashv R$ are Quillen adjunctions with respect to these model structures. 

It remains to show that a morphism $f$ in $\mathcal{A}$ is a cofibration in this model structure on $\mathcal{A}$ if and only if $If$ is a cofibration in $\mathcal{C}$. Necessity follows from the foregoing fact that the functor $I$ is left Quillen, while sufficiency follows from the fact that $L$ is left Quillen and the assumption that $I$ is fully faithful.
\end{proof}

\begin{remark}
One can show (by using \cite[Corollary 2.7]{gkr} and arguing as in the proof above) that the conclusion of Proposition \ref{restprop} also holds under the alternative hypothesis that $\mathcal{C}$ is an \emph{accessible} model category (in the sense of \cite{MR3638359}). 
\end{remark}

\subsection{Two model structures on the factorisation category $(A+B)/\sSet/(A\star B)$} \label{twomodstr}
Let $A$ and $B$ be a pair of simplicial sets. 
We shall use Proposition \ref{restprop} to induce the ``Joyal'' and ``ambivariant'' model structures on $\Cyl(A,B)$ from the following two model structures on the category $(A+B)/\sSet/(A\star B)$ of factorisations\footnote{Beware that an object of the factorisation category $(A+B)/\sSet/(A\star B)$ is not an arbitrary composable pair of morphisms $A + B \lra X \lra A \star B$, but rather a pair whose composite is the canonical inclusion $A + B \lra A \star B$.} in $\sSet$ of the canonical inclusion $A + B \lra A \star B$.

The first of these model structures on the factorisation category $(A+B)/\sSet/(A\star B)$ is the one created by the forgetful functor 
\begin{equation*}
(A+B)/\sSet/(A\star B) \lra \sSet
\end{equation*}
 from the Joyal  model structure for quasi-categories  on $\sSet$.  The existence of this created model structure follows from \cite[\S2]{hirschhornerrata} (which corrects \cite[Theorem 7.6.5(3)]{MR1944041}).

The second model structure on $(A+B)/\sSet/(A\star B)$ is induced from the parametrised Joyal model structure on $\sSet/(A \star B)$ (introduced in Appendix \ref{secparajoyal}, see Theorem \ref{parajoyal}) as follows. By \cite[Theorem 7.6.5(1)]{MR1944041} -- applied to the inclusion $A +B \lra A \star B$ as an object of the category $\sSet/(A \star B)$ equipped with the parametrised Joyal model structure -- there exists a model structure on $(A+B)/\sSet/(A\star B)$ created by the  forgetful functor 
\begin{equation*}
(A+B)/\sSet/(A\star B) \lra \sSet/(A \star B)
\end{equation*}
 from the parametrised Joyal model structure on $\sSet/(A\star B)$.

Since both the Joyal model structure on $\sSet$ and the parametrised Joyal model structure on $\sSet/(A \star B)$ are cofibrantly generated, it follows by \cite{hirschhornunderover} that both of the above model structures on the factorisation category $(A+B)/\sSet/(A\star B)$ are cofibrantly generated.

\subsection{$\Cyl(A,B)$ as a reflective and coreflective subcategory of $(A + B)/\sSet/(A \star B)$} \label{bireflsec}
Recall that the disjoint union $A+B$ and join $A \star B$ are the initial and terminal objects respectively of the category $\Cyl(A,B)$. As observed in \cite[\S14.6]{joyalnotes}, the forgetful functor $\Cyl(A,B) \lra \sSet$ lifts to a fully faithful functor $\Cyl(A,B) \lra (A+B)/\sSet/(A\star B)$. This full embedding has both a left adjoint, given by the composite 
\begin{equation*}
\cd{
(A + B)/\sSet/(A \star B) \ar[r] & \sSet/(A \star B) \ar[r]^-L & \Cyl(A,B)
}
\end{equation*}
of the forgetful functor and  the reflection $L$ described in \S\ref{reflsec} below, and a right adjoint, given by the composite
\begin{equation*}
\cd{
(A + B)/\sSet/(A \star B) \ar[r]& (A + B)/\sSet \ar[r]^-R & \Cyl(A,B)
}
\end{equation*}
of the (other) forgetful functor and the coreflection $R$ described in \S\ref{coreflsec} below.

\subsection{$\Cyl(A,B)$ as a reflective subcategory of $\sSet/(A \star B)$} \label{reflsec}
As described in \cite[Remark 3.5]{stevbifib}, the fully faithful functor $\Cyl(A,B) \lra \sSet/(A \star B)$ has a left adjoint $L$, which sends an object $X \lra A \star B$ of $\sSet/(A\star B)$ 
to the $(A,B)$-cylinder $L(X)$ defined by the pushout below,
\begin{equation*}
\cd{
\partial_0X + \partial_1X \ar[r] \ar[d] & A + B \ar[d] \\
X \ar[r] & L(X) \fatpushoutcorner
}
\end{equation*}
with the induced structure map $L(X) \lra A \star B \lra \Delta[1]$. It follows that a morphism of $(A,B)$-cylinders is a monomorphism in $\Cyl(A,B)$ if and only if  its underlying morphism of simplicial sets is a monomorphism. The essential image of the full embedding $\Cyl(A,B) \lra \sSet/(A \star B)$ consists of those morphisms $X \lra A \star B$ whose pullback along the inclusion $A + B \lra A \star B$ is an isomorphism.

 \subsection{$\Cyl(A,B)$ as a coreflective subcategory of $(A+B)/\sSet$} \label{coreflsec}
The fully faithful functor $\Cyl(A,B) \lra (A+B)/\sSet$ 
has a right adjoint $R$ given by the ``quasi-categorical collage construction'' defined in \cite[\S F.5]{rvbook}. (We will not need to know anything about this right adjoint beyond its existence.)  
The essential image of the full embedding $\Cyl(A,B) \lra (A+B)/\sSet$ consists of those cospans of simplicial sets 
\begin{equation*}
\cd{
A \ar[r] & C & \ar[l] B
}
\end{equation*}
whose left leg $A \lra C$ is a sieve inclusion and whose right leg $B \lra C$ is the complementary cosieve inclusion (in the sense of \cite[\S7.2]{joyalbarcelona} and \cite[\S14.5]{joyalnotes}).

\subsection{Restricted model structures on $\Cyl(A,B)$} The following proposition 
verifies the necessary and  sufficient condition of Proposition \ref{restprop} in our two cases of interest.

\begin{proposition} \label{reflprop}
The reflection $L \colon \sSet/(A \star B) \lra \Cyl(A,B)$ preserves monomorphisms, inner anodyne extensions, and weak categorical equivalences, and inverts any morphism between objects of $\sSet/(A\star B)$ whose structure maps factor through the inclusion $A + B \lra A \star B$.
\end{proposition}
\begin{proof}
By definition (see \S\ref{reflsec}), the functor $L$ sends a morphism $f \colon X \lra Y$ in $\sSet/(A \star B)$ to the morphism $L(X) \lra L(Y)$ induced by pushout from the diagram below, in which the left-pointing maps are monomorphisms and the left-hand square is a pullback.
\begin{equation*}
\cd{
X \ar[d]_-f & \fatleftpullbackcorner \partial_0X + \partial_1X \ar[l] \ar[r] \ar[d]^-{\partial_0f + \partial_1f} & A + B \ar@{=}[d] \\
Y & \partial_0Y + \partial_1Y \ar[l] \ar[r] & A + B
}
\end{equation*}
It follows from the exactness of pushouts of monomorphisms in the presheaf category $\sSet$ that $L$ preserves monomorphisms. If $f$ is inner anodyne, then \cite[Lemma 3.21]{joyalbarcelona} implies that the morphism $\partial_0f + \partial_1f$ is inner anodyne, whence  \cite[Lemma 2.5]{MR3812459}\footnote{This lemma is a corollary of Stevenson's theorem \cite[Theorem 1.5]{MR3812459}, for which we have given a new proof in Appendix \ref{secparajoyal} (see Theorem \ref{stevthm}).} implies that $Lf$ is inner anodyne. It follows similarly from \cite[Corollary 7.11]{joyalbarcelona} (see also Proposition \ref{joyalcor}) and the gluing lemma (see \cite{reedy}) that $L$ preserves weak categorical equivalences.

Any object $(X,p)$ of $\sSet/(A \star B)$ whose structure map $p \colon X \lra A \star B$ factors through the inclusion $A + B \lra A \star B$ is sent by $L$ to the initial $(A,B)$-cylinder $A + B$. Hence $L$ inverts any morphism between two such objects.
\end{proof}

We are now ready to construct the two model structures on $\Cyl(A,B)$ described in \S\ref{secintro}, and thus complete step (1) of the proof of Joyal's conjecture outlined in \S\ref{proofoutline}. The existence of the first model structure was stated by Joyal \cite[\S14.6]{joyalnotes}; an alternative proof of its existence is given in \cite[Theorem 3.9]{stevbifib}.

\begin{theorem}[the Joyal model structure on $\Cyl(A,B)$] \label{joyalcylmod}
There exists a model structure on $\Cyl(A,B)$ in which a morphism is a cofibration, weak equivalence, or fibration if and only if its underlying morphism of simplicial sets is a cofibration, weak equivalence, or fibration respectively in the Joyal model structure for quasi-categories on $\sSet$.
\end{theorem}
\begin{proof}
This is proven by an application of Proposition \ref{restprop} to the first model structure on the factorisation category $(A+B)/\sSet/(A \star B)$ described in \S\ref{twomodstr} and the adjoint triple described in  \S\ref{bireflsec}. The necessary and sufficient condition of Proposition \ref{restprop} follows from the fact, proved in Proposition \ref{reflprop}, that the reflection $L \colon \sSet/(A \star B) \lra \Cyl(A,B)$ preserves monomorphisms and weak categorical equivalences.
\end{proof}

We call the model structure of Theorem \ref{joyalcylmod} the \emph{Joyal model structure} on $\Cyl(A,B)$. Note that (by the remark on monomorphisms made in \S\ref{reflsec}) a morphism in $\Cyl(A,B)$ is a cofibration in this model structure if and only if it is a monomorphism in $\Cyl(A,B)$. Since, for any cylinder $X \in \Cyl(A,B)$, the canonical morphism $A +B \lra X$ is a monomorphism,  every object of $\Cyl(A,B)$ is cofibrant in this model structure; the same is true of the following model structure.

\begin{theorem}[the ambivariant model structure on $\Cyl(A,B)$] \label{ambicylmod}
There exists a model structure on $\Cyl(A,B)$ whose cofibrations are the monomorphisms and whose fibrant objects are those cylinders $X \in \Cyl(A,B)$ for which the canonical morphism $X \lra A \star B$ is an inner fibration. A morphism between fibrant objects in $\Cyl(A,B)$ is a fibration if and only if it is an inner fibration. 
\end{theorem}
\begin{proof}
The first statement is proven by an application of Proposition \ref{restprop} to the second model structure on $(A + B)/\sSet/(A \star B)$ described in \S\ref{twomodstr} and the adjoint triple described in \S\ref{bireflsec}, which implies that there exists a model structure on $\Cyl(A,B)$ created by the forgetful functor $\Cyl(A,B) \lra \sSet/(A \star B)$ from the parametrised Joyal model structure on $\sSet/(A \star B)$. The necessary and sufficient condition of Proposition \ref{restprop} follows from Proposition \ref{recogprinc} and Proposition \ref{reflprop}, the latter of which implies that the reflection $L$ preserves monomorphisms and inner anodyne extensions, and inverts the projection $\mathrm{pr} \colon (J,x \circ \mathrm{pr}) \lra (\Delta[0],x)$ in $\sSet/(A \star B)$ for every $0$-simplex $x$ of $A \star B$.

By the construction of this model structure and by Theorem \ref{parajoyal}, a morphism between fibrant objects in $\Cyl(A,B)$ is a fibration if and only if it is a fibrewise isofibration over $A \star B$ (see Definition \ref{fibisodef}). But, for any cylinder $X \in \Cyl(A,B)$, the fibre of the canonical morphism $X \lra A \star B$ over any $0$-simplex of $A \star B$ is the terminal simplicial set. Hence a morphism between fibrant objects in $\Cyl(A,B)$ is a fibrewise isofibration if and only if it is an inner fibration. This proves the second statement.
\end{proof}

 We call the model structure of Theorem \ref{ambicylmod} the \emph{ambivariant model structure} on $\Cyl(A,B)$. We say that an object of $\Cyl(A,B)$ is \emph{ambifibrant} if it is a fibrant object in this model structure, and that a morphism in $\Cyl(A,B)$ is an \emph{ambivariant equivalence} if it is a weak equivalence in this model structure. 

The following proposition gives a recognition principle for left Quillen functors from the ambivariant model structure on $\Cyl(A,B)$.

\begin{proposition} \label{recogprop2}
Let $\mathcal{C}$ be a model category, and let $F \colon \Cyl(A,B) \lra \mathcal{C}$ be a cocontinuous functor that sends monomorphisms to cofibrations. Then $F$ sends every ambivariant equivalence in $\Cyl(A,B)$ to a weak equivalence in $\mathcal{C}$ if and only if it sends every inner anodyne extension in $\Cyl(A,B)$ to a weak equivalence in $\mathcal{C}$.
\end{proposition}
\begin{proof}
This follows by a standard argument from Proposition \ref{reflprop}, Theorem \ref{ambicylmod},  \cite[Lemma 7.14]{MR2342834},  and \cite[Proposition 7.15]{MR2342834}. 
\end{proof}

We conclude this section with an observation about duality, which will help to simplify the exposition of the following two sections.

\begin{observation}[duality and the ambivariant model structure] \label{dualobs}
Recall that the category of simplicial sets bears an involution $(-)^\mathrm{op} \colon \sSet \lra \sSet$ (induced from the non-trivial involution on the category $\Delta$), which sends a simplicial set $X$ to its \emph{opposite} $X^\mathrm{op}$. For each pair of simplicial sets $A$ and $B$, this involution defines an isomorphism
\begin{equation*}
\Cyl(A,B) \cong \Cyl(B^\mathrm{op},A^\mathrm{op})
\end{equation*}
between the category of $(A,B)$-cylinders and the category of $(B^\mathrm{op},A^\mathrm{op})$-cylinders. 
Since the involution $(-)^\mathrm{op} \colon \sSet \lra \sSet$ preserves inner fibrations, it follows that this isomorphism respects the ambivariant model structures on these two categories. 
\end{observation}

\section{Cylinders as presheaves} \label{pshsec}
In this section, we make Joyal's observation (see \S\ref{psh}) that the category $\Cyl(A,B)$ of $(A,B)$-cylinders is equivalent to the category of presheaves over $\Delta/A \times \Delta/B$ the basis for a deeper analysis of the ambivariant model structure on $\Cyl(A,B)$ introduced in \S\ref{secmodstrcyl}. 
 We use the related equivalence of categories 
\begin{equation*}
\Cyl(A,B) \simeq [(\Delta/A)^\mathrm{op},\sSet/B]
\end{equation*}
to construct  (Proposition \ref{reedyprop}) a \emph{Reedy model structure} on $\Cyl(A,B)$, induced from the \emph{covariant model structure} on $\sSet/B$. We prove (Theorem \ref{bousthm}) that the ambivariant model structure on $\Cyl(A,B)$ is a \emph{Bousfield localisation} of this Reedy model structure, for which the \emph{local objects} are those $(A,B)$-cylinders whose corresponding functors $(\Delta/A)^\mathrm{op} \lra \sSet/B$ send the \emph{final vertex inclusions} in $\Delta/A$ to covariant equivalences in $\sSet/B$. (A dual result relates the ambivariant model structure on $\Cyl(A,B)$ to the \emph{contravariant} model structure on $\sSet/A$ and the \emph{initial} vertex inclusions in $\Delta/B$; see Remark \ref{dualrmk}.) We will use  this result to prove the main theorem (Theorem \ref{pfqequiv}) of \S\ref{basechangesec}. 

\subsection{$\Cyl(A,B)$ as a presheaf category} \label{psh}
We begin with Joyal's observation that, for each pair of simplicial sets $A$ and $B$, the category $\Cyl(A,B)$ of $(A,B)$-cylinders is equivalent to the category of presheaves over the product $\Delta/A \times \Delta/B$ of the categories of simplices of $A$ and $B$. 
\begin{equation} \label{pivequiv}
\Cyl(A,B) \simeq [(\Delta/A \times \Delta/B)^\mathrm{op},\Set] 
\end{equation}
Under this equivalence, an $(A,B)$-cylinder $X$ with structure map $p \colon X \lra A \star B$ corresponds to the presheaf on $\Delta/A \times \Delta/B$ whose value at the object $(([m],\alpha),([n],\beta)) \in \Delta/A \times \Delta/B$ is the set of $(m+1+n)$-simplices $\sigma$ of $X$ such that $p(\sigma) = \alpha \star \beta$, i.e.\ such that the diagram below commutes.
\begin{equation*}
\cd{
& X \ar[d]^-p \\
\Delta[m] \star \Delta[n] \ar[ur]^-{\sigma} \ar[r]_-{\alpha \star \beta} & A \star B
}
\end{equation*} We refer the reader to \cite[Chapter 7]{joyalbarcelona} for further details of this equivalence.

Combining the equivalence (\ref{pivequiv}) with the standard equivalences
\begin{equation} \label{steq}
\sSet/C \simeq [(\Delta/C)^\mathrm{op},\Set],
\end{equation}
 we obtain further equivalences between $\Cyl(A,B)$ and the functor categories displayed below.
\begin{equation} \label{funeq}
[(\Delta/A)^\mathrm{op},\sSet/B] \simeq \Cyl(A,B) \simeq [(\Delta/B)^\mathrm{op},\sSet/A]
\end{equation}
(Note that, in the sequel, we shall use the word \emph{vertical} (resp.\ \emph{horizontal}) for those properties of $(A,B)$-cylinders that are most naturally expressed in terms of their corresponding functors $(\Delta/A)^\mathrm{op} \lra \sSet/B$ (resp.\ $(\Delta/B)^\mathrm{op} \lra \sSet/A$), etc.)

These equivalences of categories enable us to  employ various standard constructions (drawn from \cite[\S7]{MR2342834}, \cite[\S3]{MR3350089}, and \cite[\S4]{MR3217884}, to which we refer the reader for further details) in our analysis of the ambivariant model structure on $\Cyl(A,B)$. Following a brief survey of these constructions, the main argument of this section begins in earnest with Proposition \ref{tfaeprop}, in which we use these constructions, together with a result of Stevenson (Lemma \ref{rightcan}), to give an alternative characterisation of the ambifibrant objects of $\Cyl(A,B)$.

\subsection{Exterior (Leibniz) products} \label{extprodsec}
Recall that we may define the \emph{exterior product} bifunctor 
\begin{equation*} \label{extprod}
\cd{
[(\Delta/A)^\mathrm{op},\Set] \times [(\Delta/B)^\mathrm{op},\Set] \ar[r]^-{\boxtimes} & [(\Delta/A \times \Delta/B)^\mathrm{op},\Set]
}
\end{equation*}
in the way made manifest by the formula $(X\boxtimes Y)_{\alpha,\beta} = X_{\alpha} \times Y_{\beta}$. 
Furthermore, we may define the corresponding \emph{exterior Leibniz product} (or \emph{exterior pushout-product}) bifunctor between arrow categories 
\begin{equation*} \label{extleibprod}
\cd{
[(\Delta/A)^\mathrm{op},\Set]^\mathbf{2} \times [(\Delta/B)^\mathrm{op},\Set]^\mathbf{2} \ar[r]^-{\widehat{\boxtimes}} & [(\Delta/A\times \Delta/B)^\mathrm{op},\Set]^\mathbf{2},
}
\end{equation*}
which sends a pair of morphisms $(f \colon M \lra N, g \colon S \lra T)$ to the pushout-corner map 
\begin{equation*}
f \mathbin{\widehat{\boxtimes}} g \colon (M \boxtimes T) \cup_{M \boxtimes S} (N \boxtimes S) \lra N \boxtimes T
\end{equation*}
of the commutative square displayed below.
\begin{equation*}
\cd[@=2.5em]{
M \boxtimes S \ar[r]^-{M \mathbin{\boxtimes} g} \ar[d]_{f \mathbin{\boxtimes} S} & M \boxtimes T \ar[d]^-{f \mathbin{\boxtimes} T} \\
N \boxtimes S \ar[r]_-{N \mathbin{\boxtimes} g} & N \boxtimes T
}
\end{equation*}

It is straightforward to show that, under the equivalences  (\ref{pivequiv}) and (\ref{steq}), 
 the exterior product bifunctor 
 corresponds to the composite bifunctor
\begin{equation*}
\cd{
\sSet/A \times \sSet/B \ar[r]^-{\star} & \sSet/(A \star B) \ar[r]^-{L} & \Cyl(A,B),
}
\end{equation*}
where $L$ denotes the reflection described in \S\ref{reflsec}. 
Furthermore, the exterior Leibniz product bifunctor  
corresponds under these equivalences to the composite bifunctor
\begin{equation*}
\cd{
(\sSet/A)^{\mathbf{2}} \times (\sSet/B)^{\mathbf{2}} \ar[r]^-{\widehat{\star}} & (\sSet/(A \star B))^{\mathbf{2}} \ar[r]^-{L^{\mathbf{2}}} & \Cyl(A,B)^{\mathbf{2}},
}
\end{equation*}
whose first factor is the \emph{Leibniz join} (or \emph{pushout-join}) bifunctor. 
We shall henceforth denote these two  composite bifunctors by $\boxtimes$ and $\widehat{\boxtimes}$ respectively, so that $(X,p) \boxtimes (Y,q) =  L(X \star Y, p \star q)$ and $f \mathbin{\widehat{\boxtimes}} g = L(f \mathbin{\widehat{\star}} g)$.

\subsection{Left and right division} \label{divsec}
The exterior product bifunctor
\begin{equation*}
\cd{
\sSet/A \times \sSet/B \ar[r]^-{\boxtimes} & \Cyl(A,B)
}
\end{equation*}
 is cocontinuous in each variable, and therefore forms part of a two-variable adjunction 
\begin{equation*}
(\sSet/B)(S,M\backslash X) \cong (\Cyl(A,B))(M \boxtimes S,X) \cong (\sSet/A)(M,X/S),
\end{equation*} 
where, following \cite[\S7]{MR2342834}, we denote\footnote{As in \cite{MR2342834}, so too in this paper is there no risk of confusing these constructions with the similarly denoted slice constructions defined in \cite[\S3]{MR1935979}, since the latter do not here appear.} the two right adjoint bifunctors
\begin{equation*} \label{divops}
\cd{
(\sSet/A)^\mathrm{op} \times \Cyl(A,B) \ar[r]^-{\backslash} & \sSet/B
}
\qquad
\cd{
\Cyl(A,B) \times (\sSet/B)^\mathrm{op}  \ar[r]^-{/} & \sSet/A
}
\end{equation*}
by the symbols of left division and right division, as displayed.

The Yoneda lemma implies the important observation that the functor $(\Delta/A)^\mathrm{op} \lra \sSet/B$ to which an object $X \in \Cyl(A,B)$ corresponds under the equivalence (\ref{funeq}) is naturally isomorphic to the composite
\begin{equation*}
\cd[@C=2.5em]{
(\Delta/A)^\mathrm{op} \ar[r] & (\sSet/A)^\mathrm{op} \ar[r]^-{-\backslash X} & \sSet/B
}
\end{equation*}
of the Yoneda embedding and the ``left division of $X$'' functor.

Furthermore, under the  equivalences of \S\ref{psh}, these ``division'' bifunctors correspond to the evident \emph{weighted limit} bifunctors, as indicated below.
\begin{equation*}
\cd{
[(\Delta/A)^\mathrm{op},\Set]^\mathrm{op} \times [(\Delta/A)^\mathrm{op},\sSet/B] \ar[r]^-{\backslash} & \sSet/B
}
\end{equation*}
\begin{equation*}
\cd{
 [(\Delta/B)^\mathrm{op},\sSet/A] \times [(\Delta/B)^\mathrm{op},\Set]^\mathrm{op} \ar[r]^-{/} & \sSet/A
}
\end{equation*}
Thus, for $M \in \sSet/A$ and $X \in \Cyl(A,B)$, the object $M\backslash X \in \sSet/B$ is the limit of the functor $X \colon (\Delta/A)^\mathrm{op} \lra \sSet/B$ weighted by the functor $M \colon (\Delta/A)^\mathrm{op} \lra \Set$. 

\begin{observation}[Leibniz joins and lifting problems] \label{leiblift}
Let $X \in \Cyl(A,B)$ with structure map $p \colon X \lra A \star B$, and consider a pair of morphisms in $\sSet/A$ and $\sSet/B$ as displayed below.
\begin{equation*}
\cd[@C=1em]{
M \ar[rr]^-f \ar[dr]_-{nf} && N \ar[dl]^-n \\
& A
}
\qquad
\qquad
\cd[@C=1em]{
S \ar[rr]^-g \ar[dr]_-{tg} && T \ar[dl]^-t \\
& B
}
\end{equation*}
By adjointness (see e.g.\ \cite[Proposition 7.6]{MR2342834}, \cite[Proposition 3.3]{MR3350089}, or  \cite[Observation 4.11]{MR3217884}) and the formula $f \mathbin{\widehat{\boxtimes}} g = L(f\mathbin{\widehat{\star}} g)$ (see \S\ref{extprodsec} above), the following are equivalent. 
\begin{enumerate}[(i)]
\item Any lifting problem in $\sSet$ of the form 
\begin{equation*}
\cd{
(M \star T) \cup_{M \star S} (N\star S) \ar[r] \ar[d]_-{f \mathbin{\widehat{\star}} g} & X \ar[d]^-p \\
N \star T \ar[r]_-{n \mathbin{\star} t} \ar@{..>}[ur] & A \star B
}
\end{equation*}
has a solution.
\item The cylinder $X$ has the right lifting property in $\Cyl(A,B)$ with respect to the exterior Leibniz product $f \mathbin{\widehat{\boxtimes}}g \colon (M \boxtimes T) \cup_{M \boxtimes S} (N \boxtimes S) \lra N \boxtimes T$ of the pair of morphisms displayed above.
\item The morphism $f\backslash X \colon N \backslash X \lra M \backslash X$ has the right lifting property in $\sSet/B$ with respect to the morphism $g \colon (S,tg) \lra (T,t)$.
\item The morphism $X/g \colon X/T \lra X/S$ has the right lifting property in $\sSet/A$ with respect to the morphism $f \colon (M,nf) \lra (N,n)$.
\end{enumerate}
\end{observation}

\begin{observation}[cellular presentations of monomorphisms] \label{cellobs}
Every monomorphism in $\sSet/A$ can be expressed as a countable composite of pushouts of coproducts of the boundary inclusions $b_m \colon (\partial\Delta[m],\partial\alpha) \lra (\Delta[m],\alpha)$ for $([m],\alpha) \in \Delta/A$. Furthermore, every monomorphism in $\Cyl(A,B)$  can be expressed as a countable composite of pushouts of coproducts of the exterior Leibniz products 
\begin{equation*}
b_m \mathbin{\widehat{\boxtimes}} b_n \colon (\Delta[m],\alpha) \boxtimes (\partial\Delta[n],\partial\beta) \cup (\partial\Delta[m],\partial\alpha) \boxtimes (\Delta[n],\beta) \lra (\Delta[m],\alpha) \boxtimes (\Delta[n],\beta)
\end{equation*} 
for $([m],\alpha) \in \Delta/A$ and $([n],\beta) \in \Delta/B$. These two statements follow from the Eilenberg--Zilber lemma applied to the categories $\Delta/A$ and $\Delta/A \times \Delta/B$ equipped with their Eilenberg--Zilber Reedy structures inherited from the standard Eilenberg--Zilber Reedy structure on $\Delta$ (see e.g.\ \cite[\S1.3]{MR3931682}).
\end{observation}

Before we can state Proposition \ref{tfaeprop}, we must first make a couple of definitions.

\begin{recall}[the covariant model structure] \label{covrecall}
There exists a model structure (due to Joyal; see \cite[Chapter 8]{joyalbarcelona}) on the slice category $\sSet/B$  whose cofibrations are the monomorphisms and whose fibrant objects are the left fibrations with codomain $B$, which we shall sometimes call the \emph{left fibrant} objects of $\sSet/B$. A morphism between left fibrant objects is a fibration if and only if it is a left fibration. This model structure is called the \emph{covariant model structure} on $\sSet/B$. A morphism in $\sSet/B$ is a \emph{covariant equivalence} (see \cite[Definition 8.2]{joyalbarcelona}) if and only if it is a weak equivalence in this model structure.  
\end{recall}

\begin{definition}[Reedy fibrant cylinders]
A cylinder $X \in\Cyl(A,B)$ is \emph{vertically Reedy left fibrant} if the morphism on the left below is a left fibration in $\sSet/B$ for every $([m],\alpha) \in \Delta/A$.
\begin{equation*}
b_m \backslash X \colon (\Delta[m],\alpha)\backslash X \lra (\partial\Delta[m],\partial\alpha)\backslash X
\qquad
\qquad
X / b_n \colon X / (\Delta[n],\beta) \lra X/(\partial\Delta[n],\partial\beta)
\end{equation*}
 Dually, a cylinder $X \in\Cyl(A,B)$ is \emph{horizontally Reedy right fibrant} if the morphism on the right above is a right fibration in $\sSet/A$ for every $([n],\beta) \in \Delta/B$.
\end{definition}

\begin{observation}[Reedy fibrant cylinders recast] \label{reedyrecast}
It follows from Observation \ref{cellobs} that a cylinder $X \in \Cyl(A,B)$ is vertically Reedy left fibrant if and only if the morphism $f\backslash X \colon N\backslash X \lra M \backslash X$ is a left fibration in $\sSet/B$ for every monomorphism $f \colon M \lra N$ in $\sSet/A$. In particular, if $X \in \Cyl(A,B)$ is vertically Reedy left fibrant, then the object $M \backslash X \in \sSet/B$ is left fibrant for every object $M \in \sSet/A$. 

Observation \ref{leiblift} implies that a cylinder $X \in \Cyl(A,B)$ is horizontally Reedy right fibrant if and only if the morphism
\begin{equation*}
h^k_m \backslash X \colon (\Delta[m],\alpha) \backslash X \lra (\Lambda^k[m],\Lambda^k[\alpha]) \backslash X,
\end{equation*}
induced by the horn inclusion $h^k_m \colon \Lambda^k[m] \lra \Delta[m]$, is a trivial fibration in $\sSet/A$ for every $m \geq 1$, $0 < k \leq m$, and $([m],\alpha)\in \Delta/A$. 
\end{observation}

For each $m \geq 1$, let $i_m \colon \Delta[0] \lra \Delta[m]$ denote the \emph{final vertex inclusion}, which picks out the final vertex $m$ of $\Delta[m]$.

\begin{definition}[vertically right local cylinders] \label{vloc}  A cylinder  $X \in \Cyl(A,B)$ is 
\emph{vertically right local} if the morphism 
\begin{equation*}
i_m\backslash X \colon (\Delta[m],\alpha)\backslash X \lra (\Delta[0],\alpha_m)\backslash X
\end{equation*}
is a covariant equivalence in $\sSet/B$ for every $m \geq 1$ and $([m],\alpha) \in \Delta/A$.
\end{definition}

Recall that a class $\mathsf{C}$ of monomorphisms in a category is said to have the \emph{right cancellation property} if 
($u \in \mathsf{C}$ and $vu \in \mathsf{C}$) $\implies$ $v \in \mathsf{C}$ 
 for any composable pair of monomorphisms $u$ and $v$.

\begin{lemma}[Stevenson] \label{rightcan}
Let $\mathsf{C}$ be a class of monomorphisms of simplicial sets which is closed under composition, stable under pushout, and has the right cancellation property. Then the following are equivalent.
\begin{enumerate}[font=\normalfont,  label=(\roman*)]
\item $\mathsf{C}$ contains the horn inclusion $\Lambda^k[m] \lra \Delta[m]$ for every $m \geq 1$ and $0 < k \leq m$.
\item $\mathsf{C}$ contains the final vertex inclusion $i_m \colon \Delta[0] \lra \Delta[m]$ for every $m \geq 1$.
\end{enumerate}
\end{lemma}
\begin{proof}
See \cite[Proposition 2.6]{MR3683375}. 
\end{proof}

\begin{proposition} \label{tfaeprop}
Let $X\in \Cyl(A,B)$. The following are equivalent.
\begin{enumerate}[font=\normalfont,  label=(\roman*)]
\item The canonical morphism $X \lra A \star B$ is an inner fibration (i.e.\ $X$ is ambifibrant).
\item $X$ is vertically Reedy left fibrant and horizontally Reedy right fibrant.
\item $X$ is vertically Reedy left fibrant and vertically right local.
\end{enumerate}
\end{proposition}
\begin{proof}
We first prove the equivalence (i) $\Longleftrightarrow$ (ii).  By definition, the canonical morphism $X \lra A \star B$ is an inner fibration if and only if every lifting problem of the form displayed below,
\begin{equation*}
\cd{
\Lambda^k[l] \ar[d] \ar[r] & X \ar[d] \\
\Delta[l] \ar@{..>}[ur] \ar[r]_-{\sigma} & A \star B
}
\end{equation*}
where $l \geq 2$ and $0 < k  < l$, has a solution. These lifting problems fall into two cases. If $\sigma \colon \Delta[l] \lra A \star B$ factors through the inclusion $A +B \lra A \star B$, then the lifting problem is solved trivially, since the pullback of $X \lra A \star B$ along $A +B \lra A \star B$ is an isomorphism (see \S\ref{reflsec}). Otherwise, $\sigma$ is of the form $\alpha \star \beta \colon \Delta[m] \star \Delta[n] \lra A \star B$ for some $([m],\alpha) \in \Delta/A$ and $([n],\beta) \in \Delta/B$, in which case the displayed inner horn inclusion is either (see \cite[Lemma 3.3]{MR1935979})  the Leibniz join 
\begin{equation*}
(\Lambda^k[m] \lra \Delta[m])\mathbin{\widehat{\star}}(\partial\Delta[n] \lra \Delta[n])
\end{equation*} if $0 < k \leq m$, or the Leibniz join 
\begin{equation*}
(\partial\Delta[m] \lra \Delta[m])\mathbin{\widehat{\star}}(\Lambda^{k-m-1}[n] \lra \Delta[n])
\end{equation*}
if $m+1 \leq k < m + 1 + n$.

It thus follows by Observation \ref{leiblift} 
 that the canonical morphism $X \lra A \star B$ is an inner fibration if and only if $(\Delta[m],\alpha) \backslash X \lra (\partial\Delta[m],\partial\alpha) \backslash X$ is a left fibration for all $([m],\alpha) \in \Delta/A$ and $X/(\Delta[n],\beta) \lra X / (\partial\Delta[n],\partial\beta)$ is a right fibration for all $([n],\beta) \in \Delta/B$, that is, if and only if $X$ is vertically Reedy left fibrant and  horizontally Reedy right fibrant.

We now prove the equivalence (ii) $\Longleftrightarrow$ (iii). Suppose that $X$ is vertically Reedy left fibrant. 
Let $\mathsf{C}$ denote the class of monomorphisms $f \colon M \lra N$ in $\sSet$ with the property that, for any morphism $n \colon N \lra A$ in $\sSet$, the induced morphism $f \backslash X \colon (N,n)\backslash X \lra (M,nf) \backslash X$ is a covariant equivalence (or equivalently a trivial fibration, since $X$ is vertically Reedy left fibrant) in $\sSet/B$. This class $\mathsf{C}$ of monomorphisms is closed under composition, stable under pushout, and has the right cancellation property. Hence Lemma \ref{rightcan} implies (via Observation \ref{reedyrecast}) that $X$ is horizontally Reedy right fibrant if and only if it is vertically right local.
\end{proof}

We now use the equivalence $\Cyl(A,B) \simeq [(\Delta/A)^\mathrm{op},\sSet/B]$ (see \S\ref{psh}) to construct a ``Reedy'' model structure on $\Cyl(A,B)$, whose weak equivalences  are described in the following definition.

\begin{definition}[vertical covariant equivalences]
A morphism $f \colon X \lra Y$ in $\Cyl(A,B)$ is  a \emph{vertical covariant equivalence} if the morphism
\begin{equation*}
(\Delta[m],\alpha)\backslash f \colon (\Delta[m],\alpha)\backslash X \lra (\Delta[m],\alpha)\backslash Y
\end{equation*}
is a covariant equivalence in $\sSet/B$ for every $([m],\alpha) \in \Delta/A$.
\end{definition}

\begin{proposition} \label{reedyprop}
There exists a model structure on the category $\Cyl(A,B)$ whose cofibrations are the monomorphisms, whose weak equivalences are the vertical covariant equivalences,  and whose fibrant objects are the vertically Reedy left fibrant objects. 
\end{proposition}
\begin{proof}
By for instance  \cite[Theorem 4.18]{MR3217884}, the Reedy structure on $\Delta/A$ (inherited from the standard Reedy structure on $\Delta$) and the covariant model structure on $\sSet/B$ (see Recollection \ref{covrecall}) cooperate to endow the functor category $[(\Delta/A)^\mathrm{op},\sSet/B]$ with a Reedy model structure, which transports along the equivalence (\ref{funeq}) to a model structure on $\Cyl(A,B)$. By construction, the weak equivalences of this model structure are  the vertical covariant equivalences;  by \cite[Corollary 6.7]{MR3217884} and Observation \ref{cellobs}, the cofibrations are the monomorphisms. 

By construction, an object $X \in \Cyl(A,B)$ is fibrant in this model structure if and only if the morphism 
\begin{equation*}
b_m\backslash X \colon (\Delta[m],\alpha)\backslash X \lra (\partial\Delta[m],\partial\alpha) \backslash X
\end{equation*}
 is a fibration between fibrant objects in the covariant model structure on $\sSet/B$ for every $([m],\alpha) \in \Delta/A$. 
 Hence, by Observation \ref{reedyrecast} and Recollection \ref{covrecall}, an object  of $\Cyl(A,B)$ is fibrant in this model structure if and only if it is vertically Reedy left fibrant.
\end{proof}

We call the model structure of Proposition \ref{reedyprop} the \emph{vertical Reedy covariant model structure} on $\Cyl(A,B)$.

The following theorem is phrased in terms of the theory of Bousfield localisations of model categories (see e.g.\ \cite[Appendix E]{joyalbarcelona}), from which we now recall a few basic definitions and results.

\begin{definition}[Bousfield localisation] \label{bousdef}
Let $\mathsf{M}$ and $\mathsf{N}$ be model structures on a category $\mathcal{C}$. The model structure $\mathsf{N}$ is  a \emph{Bousfield localisation} of the model structure $\mathsf{M}$ if $\mathsf{N}$ has the same class of cofibrations as $\mathsf{M}$, and if every fibrant object of $\mathsf{N}$ is fibrant in $\mathsf{M}$.
\end{definition}

\begin{definition}[local objects]
Let $\mathsf{M}$ and $\mathsf{N}$ be model structures on a category $\mathcal{C}$, and suppose that $\mathsf{N}$ is a Bousfield localisation of $\mathsf{M}$. 
An object of $\mathcal{C}$ is  \emph{local} (with respect to this Bousfield localisation) if it is weakly equivalent in $\mathsf{M}$ to a fibrant object of $\mathsf{N}$.
\end{definition}

\begin{lemma} \label{bouslem}
Let $\mathsf{M}$ and $\mathsf{N}$ be model structures on a category $\mathcal{C}$, and suppose that $\mathsf{N}$ is a Bousfield localisation of $\mathsf{M}$. Every weak equivalence in $\mathsf{M}$ is a weak equivalence in $\mathsf{N}$. Conversely, any weak equivalence in $\mathsf{N}$ between local objects is a weak equivalence in $\mathsf{M}$.
\end{lemma} 
\begin{proof}
See e.g.\ \cite[Proposition E.1.10]{joyalbarcelona} and \cite[Proposition E.2.21]{joyalbarcelona}.
\end{proof}

We now use Proposition \ref{tfaeprop} to prove the main result of this section.

\begin{theorem} \label{bousthm}
On the category $\Cyl(A,B)$, the ambivariant model structure is a Bousfield localisation of the vertical Reedy covariant model structure. An object of $\Cyl(A,B)$ is local with respect to this Bousfield localisation if and only if it is vertically right local.
\end{theorem}
\begin{proof}
In both model structures, the cofibrations  are precisely the monomorphisms. By Proposition \ref{tfaeprop}, every ambifibrant object of $\Cyl(A,B)$ is vertically Reedy left fibrant. 
This proves the first statement. 

As indicated by the diagram below, it is immediate from the definitions and the two-out-of-three property  that an object of $\Cyl(A,B)$ is vertically right local if it is weakly equivalent in the vertical Reedy covariant model structure to a vertically right local object. 
\begin{equation*}
\cd[@C=2.5em]{
(\Delta[m],\alpha)\backslash X \ar[r]^-{i_m\backslash X} \ar@{-}[d]_-{\sim} & (\Delta[0],\alpha_m) \backslash X \ar@{-}[d]^-{\sim} \\
(\Delta[m],\alpha)\backslash Y \ar[r]_-{i_m\backslash Y} & (\Delta[0],\alpha_m) \backslash Y
}\end{equation*}
Hence, to prove the second statement, it suffices to prove that a vertically Reedy left fibrant object of $\Cyl(A,B)$ is ambifibrant if and only if it is vertically right local. This was shown in Proposition \ref{tfaeprop}.
\end{proof}

\begin{corollary} \label{bouscor}
In the category $\Cyl(A,B)$,  every vertical covariant equivalence is an ambivariant equivalence. Conversely, any ambivariant equivalence in $\Cyl(A,B)$ between vertically right local objects is a vertical covariant equivalence.
\end{corollary}
\begin{proof}
This follows from Theorem \ref{bousthm} by Lemma \ref{bouslem}.
\end{proof}

This completes the main argument of this section. We conclude this section with a few remarks which, though they will play no role in the sequel, may be of interest to the reader.

\begin{remark}[dual results] \label{dualrmk}
By Observation \ref{dualobs}, each of the results of this section has a dual result.  For example, to Proposition \ref{tfaeprop} we may add the further equivalent condition:
\begin{enumerate}[(iv)]
\item $X$ is horizontally Reedy right fibrant and horizontally left local,
\end{enumerate}
where the latter property means that the morphism 
\begin{equation*}
X/i_0 \colon X / (\Delta[n],\beta) \lra X / (\Delta[0],\beta_0),
\end{equation*}
induced by the \emph{initial} vertex inclusion $i_0 \colon \Delta[0] \lra \Delta[n]$, is a \emph{contravariant} equivalence in $\sSet/A$ for every $n \geq 1$ and $([n],\beta) \in \Delta/B$.
\end{remark}

\begin{remark}[cylinders and bisimplicial sets] \label{cylbis}
(In this remark, we let the symbol $\boxtimes$ denote the exterior product bifunctor $\boxtimes \colon \sSet \times \sSet \lra \ssSet$.)

As observed by Joyal in \cite[Chapter 7]{joyalbarcelona}, the equivalence (\ref{pivequiv}) yields an equivalence between $\Cyl(A,B)$ and the category $\ssSet/(A \boxtimes B)$ of bisimplicial sets over the exterior product $A \boxtimes B$. Under this equivalence, the ambifibrant objects of $\Cyl(A,B)$ correspond to the ambifibrations with codomain $A \boxtimes B$, where a morphism of bisimplicial sets is said to be an \emph{ambifibration} if it has the right lifting property with respect to the exterior Leibniz products $h^k_m \mathbin{\widehat{\boxtimes}} b_n$ (for $m \geq 1$, $0 < k \leq m$, $n \geq 0$) and $b_m \mathbin{\widehat{\boxtimes}} h_n^k$ (for $m \geq 0$, $n \geq 1$, $0 \leq k < n$). 

Note that the category of bisimplicial sets bears an involution $(-)^\mathrm{vop}$, induced from the involution $(-)^\mathrm{op} \times \mathrm{id}$ on $\Delta \times \Delta$, which sends a bisimplicial set $X$ to its \emph{vertical opposite} $X^\mathrm{vop}$. Note also that $(A \boxtimes B)^\mathrm{vop} = A^\mathrm{op} \boxtimes B$. This involution on $\ssSet$ interchanges the class of ambifibrations and the class of \emph{left bifibrations} defined in \cite[Definition 5.5.10]{MR3931682}. It follows that, under the composite equivalence
\begin{equation*}
\cd[@C=3em]{
\Cyl(A,B) \ar[r]_-{\simeq} & \ssSet/(A \boxtimes B) \ar[r]_-{\cong}^-{(-)^\mathrm{vop}} & \ssSet/(A^\mathrm{op} \boxtimes B),
}
\end{equation*}
the ambivariant model structure on $\Cyl(A,B)$ corresponds to the \emph{bicovariant} model structure on $\ssSet/(A^\mathrm{op} \boxtimes B)$ constructed in \cite[Theorem 5.5.13]{MR3931682}.
\end{remark}

\begin{remark}[alternative constructions of the ambivariant model structure]
The results of this section suggest a few alternative constructions of the ambivariant model structure on $\Cyl(A,B)$. To give just two such alternatives, 
we could construct it as the Bousfield localisation of the vertical Reedy covariant model structure of Proposition \ref{reedyprop} at the set of exterior Leibniz products
\begin{equation*}
 h^k_m \mathbin{\widehat{\boxtimes}} b_n \colon (\Delta[m],\alpha) \boxtimes (\partial\Delta[n],\partial\beta) \cup (\Lambda^k[m],\Lambda^k[\alpha]) \boxtimes (\Delta[n],\beta) \lra (\Delta[m],\alpha) \boxtimes (\Delta[n],\beta)
 \end{equation*}
 for $m \geq 1$, $0 < k \leq m$, $n \geq 0$, $([m],\alpha) \in \Delta/A$, and $([n],\beta) \in \Delta/B$  (or at yet some other set of morphisms), or  by transporting the bicovariant model structure on $\ssSet/(A^\mathrm{op} \mathbin{\boxtimes} B)$ along the composite equivalence displayed in Remark \ref{cylbis}.  

Despite these alternatives, we have chosen to give the construction presented in \S\ref{secmodstrcyl} because it arrives directly at the defining properties of the ambivariant model structure which make it central to our proof of Joyal's conjecture (viz.\ it has the same cofibrations as the Joyal model structure on $\Cyl(A,B)$ and its (fibrations between) fibrant objects are precisely as described in Joyal's conjecture), because it parallels the natural construction of the Joyal model structure on $\Cyl(A,B)$, and because it uses, and has thus provided an opportunity to expose, the parametrised Joyal model structures introduced in Appendix \ref{secparajoyal}, which are of independent interest.
\end{remark}

\section{Change of base} \label{basechangesec}
The goal of this section is to carry out step (\ref{stepfour}) of the proof of Joyal's conjecture outlined in \S\ref{proofoutline}. We prove (Proposition \ref{pfqadj}) that, for each pair of morphisms of simplicial sets $u$ and $v$, the \emph{pushforward--pullback adjunction} $(u,v)_! \dashv (u,v)^*$ (see \S\ref{ppadjsec})  
is a Quillen adjunction with respect to both the Joyal and ambivariant model structures (constructed in \S\ref{secmodstrcyl}). Furthermore, we use the results of \S\ref{pshsec} to prove (Theorem \ref{pfqequiv}) that, if $u$ and $v$ are weak categorical equivalences, then the pushforward--pullback adjunction $(u,v)_! \dashv (u,v)^*$ is a Quillen equivalence with respect to the ambivariant model structures.

\subsection{The pushforward--pullback adjunction} \label{ppadjsec}
Let $u \colon A \lra A'$ and $v \colon B \lra B'$ be a pair of morphisms of simplicial sets. Recall from \cite[\S14.6]{joyalnotes} the adjunction
\begin{equation*} \label{pushforwardadj}
\xymatrix{
\Cyl(A,B) \ar@<+1.5ex>[rr]_-{\hdash}^-{(u,v)_!} && \ar@<+1.5ex>[ll]^-{(u,v)^*} \Cyl(A',B')
}
\end{equation*}
whose left adjoint sends an $(A,B)$-cylinder $X$ to the \emph{pushforward} $(A',B')$-cylinder $(u,v)_!(X)$ defined by the pushout square on the left below,
\begin{equation*}
\cd{
A + B \ar[r]^-{u + v} \ar[d] & A' + B' \ar[d] \\
X \ar[r] & \fatpushoutcorner (u,v)_!(X)
}
\qquad
\qquad
\cd{
(u,v)^*(Y) \ar[r] \ar[d] \fatpullbackcorner & Y \ar[d] \\
A \star B \ar[r]_-{u \star v} & A' \star B'
}
\end{equation*}
and whose right adjoint sends an $(A',B')$-cylinder $Y$ to the \emph{pullback} $(A,B)$-cylinder $(u,v)^*(Y)$ defined by the pullback square on the right above.

\begin{proposition} \label{pfqadj}
The pushforward--pullback adjunction
\begin{equation*}
\xymatrix{
\Cyl(A,B) \ar@<+1.5ex>[rr]_-{\hdash}^-{(u,v)_!} && \ar@<+1.5ex>[ll]^-{(u,v)^*} \Cyl(A',B')
}
\end{equation*} is a Quillen adjunction between the Joyal (resp.\ ambivariant) model structures 
on $\Cyl(A,B)$ and $\Cyl(A',B')$. In particular, the pushforward functor $(u,v)_! \colon \Cyl(A,B) \lra \Cyl(A',B')$ preserves weak categorical equivalences and ambivariant equivalences.
\end{proposition}
\begin{proof}
It suffices to show that the left adjoint $(u,v)_!$ preserves monomorphisms, weak categorical equivalences, and (by Proposition \ref{recogprop2}) inner anodyne extensions. The proof is a reiteration of the proof of Proposition \ref{reflprop}. 
\end{proof}

We now use the results and constructions of \S\ref{pshsec} to deduce the main theorem of this section from the following theorem of Joyal.

\begin{observation} \label{pfkan}
Consider the special case of the pushforward--pullback adjunction in which $u$ is the identity morphism $1_A$.   
Under the equivalences (\ref{funeq}), the pushforward--pullback adjunction $(1_A,v)_! \dashv (1_A,v)^*$ corresponds  to the adjunction 
\begin{equation*}
\xymatrix{
[(\Delta/A)^\mathrm{op},\sSet/B] \ar@<+1.5ex>[rr]_-{\hdash}^-{[1,v_!]} && \ar@<+1.5ex>[ll]^-{[1,v^*]} [(\Delta/A)^\mathrm{op},\sSet/B'],
}
\end{equation*}
which is induced from the adjunction 
\begin{equation*}
\xymatrix{
\sSet/B \ar@<+1.5ex>[rr]_-{\hdash}^-{v_!} && \ar@<+1.5ex>[ll]^-{v^*} \sSet/B'
}
\end{equation*}
whose left and right adjoints are given by composition with and pullback along the morphism $v \colon B \lra B'$ respectively. 
Hence, by \S\ref{divsec}, there exists an isomorphism
\begin{equation} \label{nattyiso}
(\Delta[m],\alpha) \backslash (1_A,v)_!(X) \cong v_!((\Delta[m],\alpha)\backslash X)
\end{equation}
in $\sSet/{B'}$, natural in  $([m],\alpha) \in \Delta/A$ and $X \in \Cyl(A,B)$.
\end{observation}

\begin{theorem}[Joyal] \label{joyalthm}
If $v \colon B \lra B'$ is a weak categorical equivalence, then the adjunction
\begin{equation*}
\xymatrix{
\sSet/B \ar@<+1.5ex>[rr]_-{\hdash}^-{v_!} && \ar@<+1.5ex>[ll]^-{v^*} \sSet/{B'}
}
\end{equation*}
is a Quillen equivalence between the covariant model structures on $\sSet/B$ and $\sSet/{B'}$.
\end{theorem}
\begin{proof}
See \cite[Theorem 10.2]{joyalbarcelona} for a proof that this adjunction is a Quillen adjunction, and see \cite[Theorem 5.2.14]{MR3931682} for a proof\footnote{This proof largely follows (and improves upon) Joyal's original, unpublished proof.} that it is moreover a Quillen equivalence.  
\end{proof}

\begin{lemma} \label{loclem}
For any simplicial set $A$ and any morphism of simplicial sets $v \colon B\lra B'$, the pushforward functor $(1_A,v)_! \colon \Cyl(A,B) \lra \Cyl(A,B')$ preserves vertically right local objects.
\end{lemma}
\begin{proof}
Let $X$ be a vertically right local object of $\Cyl(A,B)$ (see Definition \ref{vloc}). 
For each object $([m],\alpha) \in \Delta/A$, the natural isomorphism (\ref{nattyiso}) implies that the morphism 
 \begin{equation*}
i_m\backslash (1_A,v)_!(X) \colon (\Delta[m],\alpha)\backslash (1_A,v)_!(X) \lra (\Delta[0],\alpha_m)\backslash (1_A,v)_!(X)
\end{equation*}
is isomorphic in $\sSet/{B'}$ to the morphism 
\begin{equation*}
v_!(i_m\backslash X) \colon v_!((\Delta[m],\alpha)\backslash X) \lra v_!((\Delta[0],\alpha_m)\backslash X).
\end{equation*}
But this latter morphism is a covariant equivalence in $\sSet/{B'}$, since $X$ is vertically right local by assumption  and since the functor 
$v_! \colon \sSet/B \lra \sSet/B'$ preserves covariant equivalences  by \cite[Theorem 10.2]{joyalbarcelona}. Hence $(1_A,v)_!(X)$ is a vertically right local object of $\Cyl(A,B')$.
\end{proof}

\begin{theorem} \label{pfqequiv}
Suppose $u \colon A \lra A'$ and $v \colon B \lra B'$ are weak categorical equivalences. Then the pushforward--pullback adjunction
\begin{equation*}
\xymatrix{
\Cyl(A,B) \ar@<+1.5ex>[rr]_-{\hdash}^-{(u,v)_!} && \ar@<+1.5ex>[ll]^-{(u,v)^*} \Cyl(A',B')
}
\end{equation*}
is a Quillen equivalence between the ambivariant model structures on $\Cyl(A,B)$ and $\Cyl(A',B')$. In particular, the pushforward functor $(u,v)_! \colon \Cyl(A,B) \lra \Cyl(A',B')$ reflects ambivariant equivalences.
\end{theorem}
\begin{proof}
Recall that the pushforward--pullback adjunction is a Quillen adjunction between the ambivariant model structures by Proposition \ref{pfqadj}. 

Now, observe that the pushforward--pullback adjunction $(u,v)_! \dashv (u,v)^*$ is isomorphic to the following composite adjunction.
\begin{equation*}
\xymatrix{
\Cyl(A,B) \ar@<+1.5ex>[rr]_-{\hdash}^-{(1_A,v)_!} && \ar@<+1.5ex>[ll]^-{(1_A,v)^*} \Cyl(A,B') \ar@<+1.5ex>[rr]_-{\hdash}^-{(u,1_{B'})_!} && \ar@<+1.5ex>[ll]^-{(u,1_{B'})^*} \Cyl(A',B')
}
\end{equation*}
Hence it suffices to consider the case in which $u$ is an identity and the case in which $v$ is an identity. We shall prove the Quillen equivalence in the first case; the second case follows by duality (see Observation \ref{dualobs}). 

We first prove that the pushforward functor $(1_A,v)_! \colon \Cyl(A,B) \lra \Cyl(A,B')$ reflects ambivariant equivalences. Let $f \colon X \lra Y$ be a morphism in $\Cyl(A,B)$ and suppose that the morphism
\begin{equation*} \label{image}
(1_A,v)_!(f) \colon (1_A,v)_!(X) \lra (1_A,v)_!(Y)
\end{equation*} 
is an ambivariant equivalence in $\Cyl(A,B')$. Since the pushforward functor $(1_A,v)_!$ preserves ambivariant equivalences by Proposition \ref{pfqadj}, we may suppose without loss of generality (by the two-out-of-three property) that $X$ and $Y$ are ambifibrant, and hence vertically right local by Proposition \ref{tfaeprop}. Hence it follows from Lemma \ref{loclem} that $(1_A,v)_!(X)$ and $(1_A,v)_!(Y)$ are vertically right local objects of  $\Cyl(A,B')$. Corollary \ref{bouscor} then implies that the morphism  $(1_A,v)_!(f)$ 
 is a vertical covariant equivalence in $\Cyl(A,B')$, i.e.\ that the morphism 
\begin{equation*}
(\Delta[m],\alpha)\backslash (1_A,v)_!(f) \colon (\Delta[m],\alpha) \backslash (1_A,v)_!(X) \lra (\Delta[m],\alpha)\backslash (1_A,v)_!(Y)
\end{equation*}
is a covariant equivalence in $\sSet/B'$ for each $([m],\alpha) \in \Delta/A$. By the natural isomorphism (\ref{nattyiso}), this latter morphism is isomorphic in $\sSet/B'$ to the morphism
\begin{equation*}
v_!((\Delta[m],\alpha)\backslash f) \colon v_!((\Delta[m],\alpha)\backslash X) \lra v_!((\Delta[m],\alpha)\backslash Y),
\end{equation*}
which is therefore also a covariant equivalence in $\sSet/B'$. 
Since the functor $v_! \colon \sSet/B \lra \sSet/{B'}$ reflects covariant equivalences by Theorem \ref{joyalthm}, we may deduce therefore that the morphism $f \colon X \lra Y$ is a vertical covariant equivalence, and hence an ambivariant equivalence in $\Cyl(A,B)$ by Corollary \ref{bouscor}. 

It remains to prove that, for each ambifibrant object $Y \in \Cyl(A,B')$, the counit morphism 
\begin{equation*}
\varepsilon_Y \colon (1_A,v)_!(1_A,v)^*(Y) \lra Y
\end{equation*}
is an ambivariant equivalence in $\Cyl(A,B')$. Indeed, this morphism is a vertical covariant equivalence (and hence an ambivariant equivalence by Corollary \ref{bouscor}), i.e., for each object $([m],\alpha) \in \Delta/A$, the morphism
\begin{equation*}
(\Delta[m],\alpha)\backslash\varepsilon_Y \colon (\Delta[m],\alpha) \backslash (1_A,v)_!(1_A,v)^*(Y) \lra (\Delta[m],\alpha)\backslash Y
\end{equation*}
is a covariant equivalence in $\sSet/{B'}$. For by Observation \ref{pfkan}, this latter morphism is isomorphic in $\sSet/{B'}$ 
 to the counit morphism
\begin{equation*}
\cd{
v_!v^*((\Delta[m],\alpha)\backslash Y) \lra (\Delta[m],\alpha) \backslash Y
}
\end{equation*} 
of the adjunction $v_! \dashv v^*$, and is therefore a covariant equivalence by Theorem \ref{joyalthm}, since $(\Delta[m],\alpha) \backslash Y$ is a left fibrant object of $\sSet/B$  by Observation \ref{reedyrecast}, Proposition \ref{tfaeprop}, and the assumption that $Y$ is ambifibrant. 
\end{proof}

\section{Joyal's cylinder conjecture} \label{joyalconjsec}
In this section, we use the  preceding results to complete the proof of Joyal's conjecture outlined in \S\ref{proofoutline}. 
We begin with a recollection from  \cite[Lemmas 3.7 and 3.8]{stevbifib} of the proof of the easy case of Joyal's conjecture. 

\begin{lemma} \label{easycase}
Suppose $A$ and $B$ are quasi-categories. Then, in the Joyal model structure on $\Cyl(A,B)$, an object $X$ is fibrant if and only if the canonical morphism $X \lra A \star B$ is an inner fibration, and a morphism between fibrant objects is a fibration if and only if it is an inner fibration.
\end{lemma}
\begin{proof}
Since every fibration in the Joyal model structure is in particular an inner fibration, it remains to prove the sufficiency of each condition.

Let $X \in \Cyl(A,B)$, and suppose that the canonical morphism $p \colon X \lra A \star B$  is an inner fibration. Since $p$ is an  inner fibration between quasi-categories (by \cite[Corollary 3.23]{joyalbarcelona}), to prove that $p$ is a fibration, it suffices to prove that the induced functor between homotopy categories $\ho(p) \colon \ho(X) \lra \ho(A \star B)$ is an isofibration. In fact, this functor is a \emph{discrete} isofibration (i.e.\ has the \emph{unique} isomorphism lifting property), since any isomorphism in the category $\ho(A \star B)$ must belong to either one of the full subcategories $\ho(A)$ or $\ho(B)$ of $\ho(A \star B)$, and the pullback of $\ho(p)$ along either full subcategory inclusion $\ho(A)    \lra \ho(A \star B)$ or $\ho(B) \lra \ho(A \star B)$ is an isomorphism  (since $X \in \Cyl(A,B)$). Hence $X$ is a fibrant object in the Joyal model structure on $\Cyl(A,B)$.

Now, let $f \colon X \lra Y$ be a morphism between fibrant objects in the Joyal model structure on $\Cyl(A,B)$, and suppose that $f$ is an inner fibration. Once again, it suffices to prove that the functor $\ho(f)$ is an isofibration. Let $q \colon Y \lra A \star B$ denote the canonical morphism from $Y$ to the terminal object of $\Cyl(A,B)$. By the previous paragraph, the functors $\ho(q)$ and $\ho(qf) = \ho(q)\circ\ho(f)$ are discrete isofibrations, whence the functor $\ho(f)$ is also a discrete isofibration. Hence $f$ is a fibration in the Joyal model structure on $\Cyl(A,B)$.
\end{proof}

\begin{remark} \label{ifonly}
Since every morphism in $\Cyl(A,B)$ is bijective on $0$-simplices, every  trivial cofibration in the Joyal model structure on $\Cyl(A,B)$ is a monic, bijective-on-$0$-simplices, weak categorical equivalence. 
Hence one can alternatively prove Lemma \ref{easycase} by using Joyal's result that any inner fibration between quasi-categories has the right lifting property with respect to all monic, bijective-on-$0$-simplices, weak categorical equivalences (cf.\ \cite[Lemma 2.19]{stevbifib}). 
\end{remark}

\begin{corollary} \label{easycor}
Suppose $A$ and $B$ are quasi-categories. Then, on the category $\Cyl(A,B)$, the Joyal model structure and the ambivariant model structure coincide. In particular, a morphism in $\Cyl(A,B)$ is a weak categorical equivalence if and only if it is an ambivariant equivalence.
\end{corollary}
\begin{proof}
The Joyal and ambivariant model structures have the same cofibrations (namely the monomorphisms) and, by Lemma \ref{easycase}, have the same fibrant objects. Hence the two model structures coincide by \cite[Proposition E.1.10]{joyalbarcelona}.
\end{proof}

We now use the results of the preceding sections to deduce the general case of Joyal's conjecture from this special case.

\begin{theorem} \label{mainthm}
Let $A$ and $B$ be a pair of simplicial sets. On the category $\Cyl(A,B)$, the Joyal model structure and the ambivariant model structure coincide.
\end{theorem}
\begin{proof}
Since the two model structures have the same cofibrations (namely the monomorphisms), it suffices to show that they have the same weak equivalences, i.e., that a morphism in $\Cyl(A,B)$ is a weak categorical equivalence if and only if it is an ambivariant equivalence. Since every fibration in the Joyal model structure is in particular an inner fibration, it follows by \cite[Proposition E.1.10]{joyalbarcelona} that every ambivariant equivalence in $\Cyl(A,B)$ is a weak categorical equivalence. It remains to prove the converse.

Let $f$ be a weak categorical equivalence in $\Cyl(A,B)$. Let $u \colon A \lra A'$ and $v \colon B \lra B'$ be weak categorical equivalences in $\sSet$ such that $A'$ and $B'$ are quasi-categories (as may be constructed by the small object argument). By Proposition \ref{pfqadj}, the morphism $(u,v)_!(f)$  in $\Cyl(A',B')$ is a weak categorical equivalence. Then, since $A'$ and $B'$ are quasi-categories, Corollary \ref{easycor} implies that $(u,v)_!(f)$ is an ambivariant equivalence in $\Cyl(A',B')$. Finally, since $u$ and $v$ are weak categorical equivalences, Theorem \ref{pfqequiv} implies that $f$ is an ambivariant equivalence in $\Cyl(A,B)$.
\end{proof}

Therefore, for each pair of simplicial sets $A$ and $B$,  everything we have proved about the ambivariant model structure on $\Cyl(A,B)$ is true also of the Joyal model structure on $\Cyl(A,B)$, for the two are one. In particular, we may deduce the following corollary.

\begin{theorem}[Joyal's cylinder conjecture] \label{conjthm}
Let $A$ and $B$ be a pair of simplicial sets. In the Joyal model structure on $\Cyl(A,B)$, an object $X$ is fibrant if and only if the canonical morphism $X \lra A \star B$ is an inner fibration, and a morphism between fibrant objects is a fibration if and only if it is an inner fibration.
\end{theorem}
\begin{proof}
Since the Joyal and ambivariant model structures on $\Cyl(A,B)$ coincide by Theorem \ref{mainthm}, they must have the same fibrant objects and the same fibrations between fibrant objects. By the description of the ambivariant model structure given in Theorem \ref{ambicylmod}, this proves the theorem.
\end{proof}

\section{Covariant equivalences} \label{luriehalfsec}
In this final section, we use our proof of Joyal's cylinder conjecture to give a new, direct proof (see Theorem \ref{luriecor}) of a characterisation of covariant equivalences due to Lurie \cite[Chapter 2]{MR2522659}, which avoids the use of the straightening theorem \cite[Theorem 2.2.1.2]{MR2522659}.

\subsection{The left cone functor}  \label{leftcone}
Let $B$ be a simplicial set. Recall from \cite[Definition 2.4.1.2]{MR2522659} the \emph{left cone functor} $C^{\triangleleft} \colon \sSet/B \lra \sSet$, which sends an object $p \colon X \lra B$ of $\sSet/B$ to its \emph{left cone}, that is, the simplicial set $C^{\triangleleft}(X,p)$ defined by the pushout square below.
\begin{equation*}
\cd{
X \ar[r]^-p \ar[d] & B  \ar[d] \\
\Delta[0] \star X \ar[r] & C^{\triangleleft}(X,p) \fatpushoutcorner
}
\end{equation*}
Observe that $C^{\triangleleft}(X,p)$ is the underlying simplicial set of a $(\Delta[0],B)$-cylinder; indeed, the left cone functor is none other than the composite
\begin{equation*}
\cd{
\sSet/B \ar[rr]^-{\Delta[0] \mathbin{\boxtimes} (-)} && \Cyl(\Delta[0],B) \ar[r] & \sSet
}
\end{equation*}
of the displayed exterior product functor (see \S\ref{extprodsec}; here $A = \Delta[0]$) and the forgetful functor.

\subsection{Lurie's characterisation of covariant equivalences}
In \cite[Chapter 2]{MR2522659}, Lurie proves that, for each simplicial set $B$, a morphism in $\sSet/B$ is a covariant equivalence if and only if it is sent by the left cone functor $C^{\triangleleft} \colon \sSet/B \lra \sSet$ to a weak categorical equivalence. 

\begin{remark} \label{htt}
While this characterisation of covariant equivalences does not appear to be stated explicitly in \cite[Chapter 2]{MR2522659}, it is nonetheless, by the following argument, a consequence of the results of that chapter. 

By \cite[Proposition 2.1.4.7]{MR2522659}, there exists a model structure on $\sSet/B$ whose cofibrations are the monomorphisms and whose weak equivalences are the morphisms in $\sSet/B$ that are sent by the composite
\begin{equation*}
\cd{
\sSet/B \ar[r]^-{C^\triangleleft} & \sSet \ar[r]^-{\mathfrak{C}} & \sSet\text{-}\Cat
}
\end{equation*}
of the left cone functor and the homotopy coherent realisation functor to Dwyer--Kan equivalences of simplicially enriched categories. 

By \cite[Corollary 2.2.3.12]{MR2522659} (proved as a corollary of the straightening theorem \cite[Theorem 2.2.1.2]{MR2522659}), the fibrant objects of this model structure on $\sSet/B$ are the left fibrations with codomain $B$. Hence, by \cite[Proposition E.1.10]{joyalbarcelona}, this model structure coincides with the covariant model structure on $\sSet/B$ (as defined and constructed in \cite[Chapter 8]{joyalbarcelona}). In particular, a morphism in $\sSet/B$ is a covariant equivalence if and only if it is sent by the composite functor displayed above to a Dwyer--Kan equivalence. 

So the characterisation follows at last from \cite[Proposition 2.2.5.8]{MR2522659} (whose proof depends on \cite[Theorem 2.4.6.1]{MR2522659} and on the straightening theorem), which states that a morphism of simplicial sets is a weak categorical equivalence if and only if it is sent by the homotopy coherent realisation functor to a Dwyer--Kan equivalence.
\end{remark}

We now use the results of the preceding sections to give a direct proof of Lurie's characterisation of covariant equivalences. We note that this proof does not depend on the straightening theorem.

\begin{theorem} \label{lastthm}
Let $B$ be a simplicial set. The adjunction
\begin{equation*}
\xymatrix{
\Cyl(\Delta[0],B) \ar@<-1.5ex>[rr]^-{\hdash}_-{\Delta[0]\backslash (-)} && \ar@<-1.5ex>[ll]_-{\Delta[0] \boxtimes (-)} \sSet/B
}
\end{equation*}
is a Quillen equivalence between the Joyal model structure on $\Cyl(\Delta[0],B)$ and the covariant model structure on $\sSet/B$.
\end{theorem}
\begin{proof}
Under the equivalence (\ref{funeq}), this adjunction corresponds to the adjunction
\begin{equation*}
\xymatrix{
[\Delta^\mathrm{op},\sSet/B] \ar@<-1.5ex>[rr]^-{\hdash}_-{\mathrm{ev}_0} && \ar@<-1.5ex>[ll]_-{\mathrm{cst}} \sSet/B
}
\end{equation*}
whose left adjoint sends an object $X$ of $\sSet/B$ to the constant simplicial object in $\sSet/B$ with value $X$, and whose right adjoint sends a simplicial object in $\sSet/B$ to its value at $[0] \in \Delta$. Moreover, by Theorems \ref{bousthm} and \ref{mainthm}, the model structure on $[\Delta^\mathrm{op},\sSet/B]$ which corresponds under this equivalence to the Joyal model structure on $\Cyl(\Delta[0],B)$ is the Bousfield localisation of the Reedy model structure (with respect to the covariant model structure on $\sSet/B$) whose local objects are the weakly constant simplicial objects. This is precisely the ``canonical model structure'' (in the sense of \cite[Theorem 3.1]{MR1833153}) on $[\Delta^\mathrm{op},\sSet/B]$ with respect to the covariant model structure on $\sSet/B$, so \cite[Theorem 3.9]{MR1833153} implies that this adjunction is a Quillen equivalence.
\end{proof}

\begin{theorem}[Lurie] \label{luriecor}
Let $B$ be a simplicial set. A morphism in $\sSet/B$ is a covariant equivalence if and only if it is sent by the left cone functor $C^{\triangleleft} \colon \sSet/B \lra \sSet$ to a weak categorical equivalence.
\end{theorem}
\begin{proof}
As observed in \S\ref{leftcone}, the left cone functor is the composite
\begin{equation*}
\cd{
\sSet/B \ar[rr]^-{\Delta[0] \boxtimes (-)} && \Cyl(\Delta[0],B) \ar[r] & \sSet
}
\end{equation*}
of the exterior product functor and the forgetful functor. Thus the result follows from Theorem \ref{lastthm}  (since every object of $\sSet/B$ is cofibrant in the covariant model structure).
\end{proof}

\begin{remark}[two definitions of covariant equivalences]
In their treatments of quasi-category theory, Joyal and Lurie give two very different definitions of covariant equivalences. Let $B$ be a simplicial set. In \cite[Definition 8.2]{joyalbarcelona}, Joyal defines a morphism $u \colon M \lra N$ in $\sSet/B$ to be a covariant equivalence if the function
\begin{equation*}
\pi_0\bHom_B(u,X) \colon \pi_0\bHom_B(N,X)  \lra \pi_0\bHom_B(M,X)
\end{equation*}
is a bijection for every left fibrant object $X$ of $\sSet/B$ (where $\bHom_B$ denotes the standard simplicial enrichment of $\sSet/B$, which is denoted by $\Fun_B$ in Recollection \ref{stanenr}). In \cite[Definition 2.1.4.5]{MR2522659}, Lurie defines a morphism in $\sSet/B$ to be a covariant equivalence if it is sent by the composite functor 
\begin{equation*}
\cd{
\sSet/B \ar[r]^-{C^\triangleleft} & \sSet \ar[r]^-{\mathfrak{C}} & \sSet\text{-}\Cat
}
\end{equation*}
to a Dwyer--Kan equivalence. 

As explained in Remark \ref{htt}, Lurie's proof of the equivalence of these two definitions uses the straightening theorem. 
Observe that this equivalence also follows directly from Theorem \ref{luriecor} and the theorem that a morphism of simplicial sets is a weak categorical equivalence if and only if it is sent by the functor $\mathfrak{C} \colon \sSet \lra \sSet\text{-}\Cat$ to a Dwyer--Kan equivalence. While Lurie's proof of the latter theorem uses the straightening theorem, alternative proofs have been given by Joyal \cite{joyalsimp} and by Dugger--Spivak \cite{MR2764043} which do not use the straightening theorem. Hence our proof of Theorem \ref{luriecor} yields a proof of the equivalence of Joyal's and Lurie's definitions of covariant equivalences which avoids the use of the straightening theorem.
\end{remark}

\appendix

\section{The parametrised Joyal model structure} \label{secparajoyal}
In this appendix, we introduce and study a new family of model structures  on the slice categories $\sSet/B$, which we call the \emph{parametrised Joyal model structures}, whose fibrant objects are the inner fibrations with codomain $B$ (see Theorem \ref{parajoyal}). We use this family of model structures to define a new class of morphisms of simplicial sets, which we call the \emph{absolute weak categorical equivalences} (see Definition \ref{acedef}), using which we prove some new results about inner anodyne extensions and inner fibrations (see \S\ref{newresults}), building on \cite{contrasample}, and give a new proof of a theorem of Stevenson (see Theorem \ref{stevthm}). 

Note that the only results of this appendix needed for the main part of this paper are  Theorem \ref{parajoyal} and Proposition \ref{recogprinc}, which are used in \S\ref{secmodstrcyl} to construct the ambivariant model structure on the category $\Cyl(A,B)$ for each pair of simplicial sets $A$ and $B$. 

\begin{remark}
The constructions and results of this section are inspired by, but not an instance of, the general theory presented in \cite[\S2.5]{MR3931682} and in the latter part of \cite[\S1.3]{MR2294028}. The existence of the parametrised Joyal model structures on the categories $\sSet/B$ is prefigured in \cite[Remark 5.1.22]{MR3931682}, where it is observed that the inner fibrations with codomain $B$ are (among the) fibrant objects in the minimal Cisinski model structure on $\sSet/B$.
\end{remark}

\subsection{Fibrewise isofibrations}
Recall that a morphism of quasi-categories $p \colon X \lra Y$ is an \emph{isofibration} if it is an inner fibration and if the induced functor between homotopy categories $\ho(p) \colon \ho(X) \lra \ho(Y)$ is an isofibration in the ordinary sense. Let $J$ denote the nerve of the contractible groupoid with two objects $0$ and $1$. It follows from Joyal's lifting theorem that an inner fibration between quasi-categories is an isofibration if and only if it has the right lifting property in $\sSet$ with respect to the end-point inclusion $\partial_0 \colon \{0\} \lra J$ (see \cite{MR1935979}). 

Furthermore, recall that a morphism between quasi-categories is a fibration in the \emph{Joyal model structure} on $\sSet$ (whose cofibrations are the monomorphisms, and whose fibrant objects are the quasi-categories) if and only if it is an isofibration (see \cite[Theorem 6.12]{joyalbarcelona}). In Theorem \ref{parajoyal} below, we shall construct, for each simplicial set $B$, a model structure on the slice category $\sSet/B$  whose cofibrations are the monomorphisms, whose fibrant objects are the inner fibrations with codomain $B$, and in which a morphism between fibrant objects is a fibration if and only if it a \emph{fibrewise isofibration} in the sense of the following definition. 

(Recall that a morphism of simplicial sets $p \colon X \lra B$ is an inner fibration if and only if the fibre $X_{\beta}$ of $p$ over each $n$-simplex $\beta$ of $B$ is a quasi-category.)

\begin{definition}[fibrewise isofibrations] \label{fibisodef}
Let $B$ be a simplicial set, and let $p \colon X \lra B$ and $q \colon Y \lra B$ be inner fibrations. A morphism $f \colon (X,p) \lra (Y,q)$ in $\sSet/B$ is a \emph{fibrewise isofibration} over $B$ if it is an inner fibration and if, for each $0$-simplex $b$ of $B$, the induced morphism between fibres $f \colon X_b \lra Y_b$ is an isofibration. 
\end{definition}

\subsection{A fundamental lemma of quasi-category theory}
We shall construct the aforementioned parametrised Joyal model structures on the slice  categories $\sSet/B$ by Cisinski's method for constructing model structures on presheaf categories presented in \cite[\S1.3]{MR2294028} and \cite[\S2.4]{MR3931682}. Our proof of this construction will use the following lemma.

\begin{lemma} \label{funlemma}
Let $j \colon S \lra T$ be a bijective-on-$0$-simplices monomorphism of simplicial sets, and suppose that $T$ is a quasi-category. Then the Leibniz product 
\begin{equation*}
\partial_0 \mathbin{\widehat{\times}} j \colon (J \times S) \cup_{\{0\} \times S} (\{0\} \times T) \lra J \times T
\end{equation*}
is an inner anodyne extension.
\end{lemma}
\begin{proof}
Since its codomain $J \times T$ is a quasi-category, it suffices to show that the morphism $\partial_0 \mathbin{\widehat{\times}} j$ has the left lifting property with respect to every inner fibration between quasi-categories $p \colon X \lra Y$. 
Since $j$ is a bijective-on-$0$-simplices monomorphism, it can be expressed as a countable composite of pushouts of coproducts of boundary inclusions $b_n \colon \partial\Delta[n] \lra \Delta[n]$ for $n \geq 1$, whence it suffices to show that the Leibniz product $\partial_0 \mathbin{\widehat{\times}} b_n$ has the left lifting property with respect to $p$ for every $n \geq 1$.  So, by adjointness,  it suffices to show that  the Leibniz cotensor $b_n \mathbin{\widehat{\pitchfork}} p \colon X^{\Delta[n]} \lra X^{\partial\Delta[n]} \times_{Y^{\partial\Delta[n]}} Y^{\Delta[n]}$, which is an inner fibration between quasi-categories by \cite[Theorem 2.18]{joyalbarcelona} (see also \cite[Corollary 3.2.8]{MR3931682}), is an isofibration for every $n \geq 1$. This  follows from the Leibniz product version of Joyal's lifting theorem (see e.g.\ the proof of \cite[Theorem 3.5.9]{MR3931682}).
\end{proof}

\begin{recall}[the standard simplicial enrichment of $\sSet/B$] \label{stanenr}
For each simplicial set $B$, the slice category $\sSet/B$ admits a standard (tensored and cotensored) enrichment over the cartesian closed category $\sSet$. The \emph{tensor} $S \otimes (X,p)$ of an object $(X,p)$ in $\sSet/B$ with a simplicial set $S$ is the object $(S \times X, p\circ \mathrm{pr}_2)$ of $\sSet/B$. For each pair of objects $(X,p)$, $(Y,q)$ in $\sSet/B$, their \emph{simplicial hom} $\Fun_B((X,p),(Y,q))$ (which we sometimes simply denote  by $\Fun_B(X,Y)$) is the simplicial set defined by the pullback square below.
\begin{equation*}
\cd{
\Fun_B((X,p),(Y,q)) \fatpullbackcorner \ar[r] \ar[d] & Y^X \ar[d]^-{q^X} \\
\Delta[0] \ar[r]_-{\ulcorner p \urcorner} & B^X
}
\end{equation*}
\end{recall}

\begin{theorem}[the parametrised Joyal model structure] \label{parajoyal}
Let $B$ be a simplicial set. There exists a cofibrantly generated model structure on the category $\sSet/B$ whose cofibrations are the monomorphisms and whose fibrant objects are the inner fibrations with codomain $B$. A morphism between fibrant objects is a fibration if and only if it is a fibrewise isofibration over $B$. 
\end{theorem}
\begin{proof}
Since the category $\sSet/B$ is equivalent to the presheaf category $[(\Delta/B)^\mathrm{op},\Set]$, we may apply the method of \cite[\S1.3]{MR2294028} and \cite[\S2.4]{MR3931682} for constructing cofibrantly generated model structures on presheaf categories whose cofibrations are the monomorphisms. Recall that one constructs such a model structure by this method by specifying an \emph{exact cylinder} (see \cite[D{\'e}finition 1.3.6]{MR2294028} or \cite[Definition 2.4.8]{MR3931682}) and a class of \emph{anodyne extensions} relative to that exact cylinder (see \cite[D{\'e}finition 1.3.10]{MR2294028} or \cite[Definition 2.4.11]{MR3931682}).

We define our exact cylinder on $\sSet/B$ to be the endofunctor $J \otimes (-)$ on $\sSet/B$ given by tensoring
(see Recollection \ref{stanenr}) with the object $J \in \sSet$; the requisite ``boundary inclusion'' and ``projection'' natural transformations are given by tensoring with the boundary inclusion $\partial J \lra J$ (where $\partial J = \{0,1\}$) and the projection $J \lra \Delta[0]$ respectively. The exact cylinder axioms DH1 and DH2 are easily verified (cf.\ \cite[Exemple 1.3.8]{MR2294028} or \cite[Example 2.4.10]{MR3931682}).  

We define our class $\mathsf{An}$ of anodyne extensions relative to this exact cylinder to be the (weak) saturation of the set of morphisms in $\sSet/B$ consisting of the end-point inclusion $\partial_0 \colon (\Delta[0],b) \lra J \otimes (\Delta[0],b)$ for every $0$-simplex $b$ in $B$, and the inner horn inclusion $h^k_n \colon (\Lambda^k[n],\Lambda^k[\beta]) \lra (\Delta[n],\beta)$  for every $n \geq 2$, $0 < k < n$, and $n$-simplex $\beta$ in $B$, as displayed below.
\begin{equation*}
\cd[@C=1em]{
\Delta[0] \ar[rr]^-{\partial_0} \ar[dr]_-{b} && J \ar[dl]^-{b \circ \mathrm{pr}} \\
& B
}
\qquad
\qquad
\cd[@C=1em]{
\Lambda^k[n] \ar[rr]^-{h^k_n} \ar[dr]_-{\Lambda^k[\beta]} && \Delta[n] \ar[dl]^-{\beta} \\
& B
}
\end{equation*}
It follows that every inner anodyne extension in $\sSet/B$ belongs to the class $\mathsf{An}$.

We now show that this class $\mathsf{An}$ satisfies the axioms for a class of anodyne extensions relative to the exact cylinder $J \otimes (-)$. Axiom An0 is immediate from the definition of the class $\mathsf{An}$. By Observation \ref{cellobs}, to verify axiom An1 it suffices to show that the Leibniz tensor
\begin{equation*}
\partial_0 \mathbin{\widehat{\otimes}} b_n \colon (J \otimes (\partial\Delta[n],\partial\beta)) \cup_{\{0\} \otimes (\partial\Delta[n],\partial\beta)} (\{0\} \otimes (\Delta[n],\beta)) \lra J \otimes (\Delta[n],\beta)
\end{equation*}
belongs to $\mathsf{An}$ for every $n \geq 0$ and $n$-simplex $\beta$ of $B$. If $n = 0$, this is immediate from the definition of $\mathsf{An}$; if $n \geq 1$, this follows from Lemma \ref{funlemma}, since every inner anodyne extension in $\sSet/B$ belongs to $\mathsf{An}$. To verify axiom An2, it suffices to show that the Leibniz tensor of each of the above generators of the saturated class $\mathsf{An}$ with the boundary inclusion $\partial J \lra J$ belongs to $\mathsf{An}$. In fact, the underlying morphism of simplicial sets of each such Leibniz tensor product is an inner anodyne extension, since it is either the Leibniz product of $\{0\} \lra J$ with the bijective-on-$0$-simplices monomorphism $\partial J \lra J$, which is inner anodyne by Lemma \ref{funlemma} (since $J$ is a quasi-category), or the Leibniz product of an inner horn inclusion with the monomorphism $\partial J \lra J$, which is inner anodyne by \cite[Theorem 2.17]{joyalbarcelona} (see also \cite[Corollary 3.2.4]{MR3931682}). 

Hence it follows from \cite[Th{\'e}or{\`e}me 1.3.22, Proposition 1.3.36]{MR2294028} or \cite[Theorem 2.4.19]{MR3931682} that there exists a cofibrantly generated model structure on $\sSet/B$ whose cofibrations are the monomorphisms, and in which an object is fibrant (resp.\ a morphism between fibrant objects is a fibration)  if and only if it has the right lifting property with respect to the class of morphisms $\mathsf{An}$. It is easily shown that these (fibrations between) fibrant objects are precisely as described in the statement of the  theorem.
\end{proof}

For each simplicial set $B$, we call the model structure of Theorem \ref{parajoyal} the \emph{parametrised Joyal model structure} on $\sSet/B$. 

\begin{proposition} \label{recogprinc}
Let $B$ be a simplicial set, let $\mathcal{C}$ be a model category, and let $F \colon \sSet/B \lra \mathcal{C}$ be a cocontinuous functor that sends monomorphisms to cofibrations. Then $F$ sends every weak equivalence in the parametrised Joyal model structure on $\sSet/B$ to a weak equivalence in $\mathcal{C}$ if and only if it sends the following morphisms to weak equivalences in $\mathcal{C}$:
\begin{enumerate}[font=\normalfont,  label=(\roman*)]
\item for each $0$-simplex $b$ of $B$, the projection $\mathrm{pr} \colon (J,b\circ \mathrm{pr}) \lra (\Delta[0],b)$, and
\item for each $n \geq 2$, $0 < k < n$, and $n$-simplex $\beta$ of $B$, the inner horn inclusion \linebreak $h^k_n \colon (\Lambda^k[n],\Lambda^k[\beta]) \lra (\Delta[n],\beta)$.
\end{enumerate}
In particular, each of the morphisms listed above is a weak equivalence in the parametrised Joyal model structure on $\sSet/B$. 
\end{proposition}
\begin{proof}
This follows by a standard argument from Theorem \ref{parajoyal}, \cite[Proposition 7.14]{MR2342834}, and \cite[Proposition 7.15]{MR2342834}. Alternatively, it follows from the construction of the model structure given in the proof of Theorem \ref{parajoyal} by \cite[Proposition 2.4.40]{MR3931682}. 
\end{proof}

\begin{observation} \label{bjday}
In the case $B =\Delta[0]$, the parametrised Joyal model structure on $\sSet/\Delta[0] \cong \sSet$ is precisely the Joyal model structure for quasi-categories.

For each simplicial set $B$, there exists (by \cite[Theorem 7.6.5(2)]{MR1944041}) a model structure  on the slice category $\sSet/B$ created by the forgetful functor $\sSet/B \lra \sSet$ from the Joyal model structure on $\sSet$. Observe that this \emph{sliced Joyal model structure} (as we shall call it) on $\sSet/B$ is a Bousfield localisation (in the sense of Definition \ref{bousdef}) of the parametrised Joyal model structure on $\sSet/B$, since every fibration in the Joyal model structure on $\sSet$ is in particular an inner fibration. Hence (by \cite[Proposition E.1.10]{joyalbarcelona}) every weak equivalence in the parametrised Joyal model structure on $\sSet/B$ is a weak categorical equivalence (see Proposition \ref{quil} for another proof of this fact). 

If $B$ is a quasi-category whose homotopy category has no non-identity isomorphisms, then every inner fibration with codomain $B$ is an isofibration, and so the ``sliced'' and ``parametrised'' Joyal model structures on $\sSet/B$ coincide (by \cite[Proposition E.1.10]{joyalbarcelona}). This coincidence does not occur in general, such as in the case $B = J$. For the morphism $\partial_0 \colon \{0\} \lra J$ is an inner fibration but not an isofibration; furthermore, this morphism is a weak categorical equivalence, but the morphism $\partial_0 \colon (\{0\},\partial_0) \lra (J,1_J)$ is not a weak equivalence in the parametrised Joyal model structure on $\sSet/J$, since it is a morphism between cofibrant-fibrant objects therein, but there exists no  morphism $(J,1_J) \lra (\{0\},\partial_0)$ in $\sSet/J$.
\end{observation}

\begin{proposition} \label{quil}
For each morphism of simplicial sets $u \colon B \lra B'$, the adjunction
\begin{equation*}
\xymatrix{
\sSet/B \ar@<1.5ex>[rr]_-{\hdash}^-{u_!} && \ar@<1.5ex>[ll]^-{u^*} \sSet/B'
}
\end{equation*}
is a Quillen adjunction between the parametrised Joyal model structures on $\sSet/B$ and $\sSet/B'$. In particular, for each simplicial set $B$,  every weak equivalence in the parametrised Joyal model structure on $\sSet/B$ is a weak categorical equivalence. 
\end{proposition}
\begin{proof}
The first statement follows from Proposition \ref{recogprinc}. The second statement follows from the first by taking $u$ to be  the morphism $B \lra \Delta[0]$.
\end{proof}

We now use Proposition \ref{recogprinc} to obtain a new proof of the following result of Joyal \cite[Corollary 7.11]{joyalbarcelona}, which we used in the proof of Proposition \ref{reflprop}. (Recall that we denote by $\partial_0$ and $\partial_1$ the functors $\sSet/\Delta[1] \lra \sSet$ that send a morphism $X \lra \Delta[1]$ to its fibres over $0$ and $1$ respectively.)

\begin{proposition}[Joyal] \label{joyalcor}
The functors $\partial_0,\partial_1 \colon \sSet/\Delta[1] \lra \sSet$ preserve weak categorical equivalences.
\end{proposition}
\begin{proof}
By Observation \ref{bjday}, a morphism in $\sSet/\Delta[1]$ is a weak equivalence in the parametrised Joyal model structure on $\sSet/\Delta[1]$ if and only if it is a weak categorical equivalence. Hence it suffices to verify the hypotheses of Proposition \ref{recogprinc}. Since the presheaf category $\sSet$ is locally cartesian closed, the functors $\partial_0$ and $\partial_1$ have both left and right adjoints, and so preserve monomorphisms and colimits. By \cite[Lemma 3.21]{joyalbarcelona}, the functors $\partial_0$ and $\partial_1$ preserve inner anodyne extensions. Hence it remains to show that these functors send the two morphisms in $\sSet/\Delta[1]$ displayed below to weak categorical equivalences.
\begin{equation*}
\cd[@C=1em]{
J \ar[rr] \ar[dr] && \{0\} \ar[dl] \\
& \Delta[1]
}
\qquad
\qquad
\cd[@C=1em]{
J \ar[rr] \ar[dr] && \{1\} \ar[dl] \\
& \Delta[1]
}
\end{equation*}
These two morphisms are sent by the functor $\partial_0$ to the projection $J \lra \Delta[0]$ and the identity $\emptyset \lra \emptyset$ respectively (and alternately by the functor $\partial_1$), which are indeed weak categorical equivalences.
\end{proof}

We now show that the parametrised Joyal model structure makes the slice category $\sSet/B$, with its standard simplicial enrichment (see Recollection \ref{stanenr}), into a ``Joyal-enriched'' model category. 

\begin{proposition} \label{enrichie}
For each simplicial set $B$, the category $\sSet/B$, equipped with the parametrised Joyal model structure and its standard simplicial enrichment, is enriched as a model category over the Joyal model structure on $\sSet$.
\end{proposition}
\begin{proof}
It suffices to show that the tensor product bifunctor
\begin{equation*}
\cd{
\sSet \times \sSet/B \ar[r]^-{\otimes} & \sSet/B
}
\end{equation*}
is a left Quillen bifunctor with respect to the Joyal model structure on $\sSet$ and the parametrised Joyal model structure on $\sSet/B$. 
This follows from two applications of Lemma \ref{recogprinc}, together with the facts that the Leibniz product in $\sSet$ of a monomorphism with a monomorphism (resp.\ an inner anodyne extension) is a monomorphism (resp.\ an inner anodyne extension), and that $J$ is an injective object of $\sSet$.
\end{proof}

\begin{observation}[the $\infty$-cosmos of $B$-parametrised quasi-categories]
It follows from \cite[Lemma 2.2.1]{MR3556697} applied to Proposition \ref{enrichie} that the full simplicial subcategory of $\sSet/B$ (with its standard simplicial enrichment, see Recollection \ref{stanenr}) whose objects are the inner fibrations with codomain $B$ is an $\infty$-cosmos (in the sense of \cite[Definition 2.1.1]{MR3556697}), in which the ``isofibrations'' are the fibrewise isofibrations over $B$ (see Definition \ref{fibisodef}). 
\end{observation}

\begin{perspective}[parametrised (quasi-)categories] \label{parass}
 Since the nerve functor $\Cat \lra \sSet$ sends every functor to an inner fibration, it is sometimes said that there is no analogue of the notion of inner fibration in ordinary category theory. Nevertheless, there are \emph{ways of seeing} ordinary functors that generalise to fruitful ways to think about inner fibrations. For example, in \cite[\S2.3.1]{MR2522659}, Lurie presents a way to think about inner fibrations based on the equivalence due to  B\'enabou between functors with codomain category $B$ and normal lax functors from $B$ to the bicategory of categories and profunctors (see \cite{stpow}). 

Following Schumacher--Street \cite{MR962317}, a closely related but more elementary way of seeing ordinary functors with codomain $B$ is  as \emph{categories parametrised by $B$} (or \emph{$B$-parametrised categories}). Seen from this point of view, a $B$-parametrised category $E$ (i.e.\ a category $E$ equipped with a functor $p \colon E \lra B$) consists of a set of objects $\ob E$ parametrised by the objects of  $B$, a family of hom-sets $E_f(x,y)$ parametrised by the morphisms of $B$ (where $E_f(x,y)$ is the fibre of the function $p \colon E(x,y) \lra B(p(x),p(y))$ above $f$), a composition operation $E_g(y,z) \times E_f(x,y) \lra E_{gf}(x,z)$ parametrised by the composition operation in $B$, and so on. 

Now let $B$ be a simplicial set. By analogy with the previous paragraph, one may think of inner fibrations with codomain $B$ as \emph{$B$-parametrised quasi-categories}. We put forth this point of view as an aid to the understanding of the parametrised Joyal model structure on $\sSet/B$ (which may be thought of as presenting the homotopy theory of $B$-parametrised quasi-categories) and of its relation to the Joyal model structure for quasi-categories on $\sSet$. For instance, this point of view naturally suggests the characterisation of equivalences of $B$-parametrised quasi-categories given in Proposition \ref{paraequiv} below.
\end{perspective}

\subsection{Equivalences of parametrised quasi-categories}
Recall that a morphism of quasi-categories is a (weak) categorical equivalence if and only if it is essentially surjective on objects and fully faithful (see e.g.\ \cite[Theorem 3.9.7]{MR3931682}). The goal of the next part of this appendix is to generalise this result to a characterisation of the weak equivalences between fibrant objects in the parametrised Joyal model structure on the category $\sSet/B$, for each simplicial set $B$.

\begin{definition}
Let $B$ be a simplicial set, and let $p \colon X \lra B$ and $q \colon Y \lra B$ be inner fibrations. A morphism $u \colon (X,p) \lra (Y,q)$ in $\sSet/B$ is said to be:
\begin{enumerate}[(i)]
\item \emph{fibrewise essentially surjective on objects} if, for every $0$-simplex $b$ of $B$, the induced morphism $u \colon X_b \lra Y_b$ is essentially surjective on objects;
\item \emph{parametrised fully faithful} if, for every pair of $0$-simplices $x,y$ in $X$, and every $1$-simplex $f \colon p(x) \lra p(y)$ in $B$, the induced morphism on hom-spaces $u \colon \Hom_{X_f}(x,y) \lra \Hom_{Y_f}(u(x),u(y))$ is an equivalence of Kan complexes.
\end{enumerate}
\end{definition}

The following lemma makes precise the sense in which the category $\sSet/B$, equipped with the parametrised Joyal model structure, is generated by the $1$-simplices of $B$ under homotopy colimits.

\begin{lemma} \label{hocolim}
Let $B$ be a simplicial set and let $\mathsf{D}$ be a class of objects of $\sSet/B$ with the following properties:
\begin{enumerate}[font=\normalfont,  label=(\roman*)]
\item $\mathsf{D}$ is saturated by monomorphisms in $\sSet/B$ \textup{(}in the sense of \cite[D{\'e}finition 1.1.12]{MR2294028} and \cite[Definition 1.3.9]{MR3931682}\textup{)},
\item any object of $\sSet/B$ weakly equivalent in the parametrised Joyal model structure to an object of $\mathsf{D}$ belongs to $\mathsf{D}$, and
\item for every $1$-simplex $f$ of $B$, the object $(\Delta[1],f)$ of $\sSet/B$ belongs to $\mathsf{D}$.
\end{enumerate}
Then every object of $\sSet/B$ belongs to $\mathsf{D}$.
\end{lemma}
\begin{proof}
Since $\mathsf{D}$ is saturated by monomorphisms, it suffices by \cite[Proposition 8.2.8]{MR2294028} or \cite[Corollary 1.3.10]{MR3931682} (applied to the category $\Delta/B$ with its standard Eilenberg--Zilber Reedy structure) to show that the object $(\Delta[n],\beta)$ belongs to $\mathsf{D}$ for every $n \geq 0$ and every $n$-simplex $\beta$ of $B$. First, we have by property (iii) that the object $(\Delta[1],f)$ belongs to $\mathsf{D}$ for every $1$-simplex $f$ of $B$. It then follows from property (i) that the object $(\Delta[0],b)$ belongs to $\mathsf{D}$ for every $0$-simplex $b$ of $B$, and hence moreover that the object $(I[n],\varphi)$ belongs to $\mathsf{D}$ for every $n \geq 2$ and every morphism $\varphi \colon I[n] \lra B$ (where $I[n]$ denotes the spine of $\Delta[n]$). Finally, since the spine inclusions $I[n] \lra \Delta[n]$ are inner anodyne for every $n \geq 2$ by \cite[Proposition 2.13]{joyalbarcelona}, and since every inner anodyne extension in $\sSet/B$ is a weak equivalence in the parametrised Joyal model structure by Proposition \ref{recogprinc}, we may deduce the result from property (ii). 
\end{proof}

\begin{proposition}[equivalences of parametrised quasi-categories] \label{paraequiv}
Let $B$ be a simplicial set, let $p \colon X \lra B$ and $q \colon Y \lra B$ be inner fibrations, and let $u \colon (X,p) \lra (Y,q)$ be a morphism in $\sSet/B$. Then the following are equivalent.
\begin{enumerate}[font=\normalfont,  label=(\roman*)]
\item The morphism $u \colon (X,p) \lra (Y,q)$ is a weak equivalence in the parametrised Joyal model structure on $\sSet/B$.
\item For every $1$-simplex $f$ of $B$, the induced morphism between fibres $u \colon X_f \lra Y_f$ is an equivalence of quasi-categories.
\item For every $1$-simplex $f$ of $B$, the morphism $$\Fun_B((\Delta[1],f),u) \colon \Fun_B((\Delta[1],f),X) \lra \Fun_B((\Delta[1],f),Y)$$ is an equivalence of quasi-categories.
\item The morphism $u \colon (X,p) \lra (Y,q)$ is fibrewise essentially surjective on objects and parametrised fully faithful. 
\end{enumerate}
\end{proposition}
\begin{proof}
The implication (i) $\Longrightarrow$ (ii) follows from an application of Proposition \ref{quil} to the morphisms $f \colon \Delta[1] \lra B$, together with the observation  that the ``sliced'' and ``parametrised'' Joyal model structure on $\sSet/\Delta[1]$ coincide (see Observation \ref{bjday}). 

The implication (ii) $\Longrightarrow$ (iii) follows from the evident natural isomorphism 
\begin{equation*}
\Fun_B((\Delta[1],f),X) \cong \Fun_{\Delta[1]}(\Delta[1],X_f)
\end{equation*}
 and the fact  that the functor   
 \begin{equation*}
 \Fun_{\Delta[1]}(\Delta[1],-) \colon \sSet/\Delta[1] \lra \sSet
 \end{equation*} is right Quillen with respect to the Joyal model structures (by \cite[Proposition 2.28]{joyalbarcelona}).

The implication (iii) $\Longrightarrow$ (i) follows by the Yoneda lemma from  an application of Lemma \ref{hocolim} to the class of objects $K\in\sSet/B$ for which the morphism 
\begin{equation*}
\Fun_B(K,u) \colon \Fun_B(K,X) \lra \Fun_B(K,Y)
\end{equation*} is an equivalence of quasi-categories.

Finally, the equivalence (ii) $\Longleftrightarrow$ (iv) follows from the fact that a morphism of quasi-categories is an equivalence if and only if it is essentially surjective on objects and fully faithful. 
\end{proof}

We prove the following proposition as a corollary of Proposition \ref{paraequiv}, though we note that it can also be proved  directly.

\begin{proposition} \label{inn2triv}
Let $p \colon X \lra B$ be an inner fibration. Then the following are equivalent.
\begin{enumerate}[font=\normalfont,  label=(\roman*)]
\item $p \colon X \lra B$ is a trivial fibration.
\item For every $1$-simplex $f$ of $B$, the pullback $X_f \lra \Delta[1]$ of $p$ along $f \colon \Delta[1] \lra B$ is a trivial fibration.
\item $p$ is surjective on objects and, for each pair of $0$-simplices $x,y$ in $X$ and each $1$-simplex $f \colon p(x) \lra p(y)$ in $B$, the hom-space $\Hom_{X_f}(x,y)$ is a contractible Kan complex.
\end{enumerate}
\end{proposition}
\begin{proof}
Since $p \colon X \lra B$ is an inner fibration, the morphism $p \colon (X,p) \lra (B,1_B)$ in $\sSet/B$ is a fibration between fibrant objects in the parametrised Joyal model structure, and hence is a weak equivalence therein if and only if it is a trivial fibration. Hence the result follows from Proposition \ref{paraequiv}.
\end{proof}

\subsection{Absolute weak categorical equivalences} \label{newresults} 
For many years, the following two questions (here suggestively posed) were open (see \cite[\S2.10]{joyalnotes}):
\begin{itemize}
\item Is a monomorphism of simplicial sets inner anodyne if and only if it is a surjective-on-$0$-simplices weak categorical equivalence?
\item Is an inner fibration a trivial fibration if and only if it is a surjective-on-$0$-simplices weak categorical equivalence?
\end{itemize}
It was recently shown by the author that the answer to both of these questions  is ``no''; an example was given in  \cite{contrasample} of a morphism of simplicial sets which is a monomorphism, an inner fibration, and a surjective-on-$0$-simplices weak categorical equivalence, but which is neither inner anodyne nor a trivial fibration.

The goal of the final part of this appendix is to show that these two statements may be corrected by replacing the property ``surjective-on-$0$-simplices weak categorical equivalence'' by the new property \emph{absolute weak categorical equivalence} (see Definition \ref{acedef}). That is, we prove 
 (see Proposition \ref{aceprop}):
\begin{itemize}
\item A monomorphism of simplicial sets is inner anodyne if and only if it is an absolute weak categorical equivalence.
\item An inner fibration is a trivial fibration if and only if it is an absolute weak categorical equivalence.
\end{itemize}
We use these results to give a new proof of a theorem of Stevenson (see Theorem \ref{stevthm}).

\begin{definition}[absolute weak categorical equivalence] \label{acedef}
A morphism of simplicial sets $f \colon X \lra Y$ is an \emph{absolute weak categorical equivalence} if, for every simplicial set $B$ and every morphism of simplicial sets $b \colon Y \lra B$, the morphism $f \colon (X,bf) \lra (Y,b)$ is a weak equivalence in the parametrised Joyal model structure on $\sSet/B$.
\end{definition}

\begin{example}  \label{aceex}
It follows from Theorem \ref{parajoyal} and Proposition \ref{recogprinc} that all inner anodyne extensions and all trivial fibrations are absolute weak categorical equivalences.
\end{example}

The following property of the class of absolute weak categorical equivalences is immediate from the definition.

\begin{proposition} \label{twothirds}
Let $u \colon A \lra B$ and $v \colon B \lra C$ be a composable pair of morphisms of simplicial sets, and suppose that $u$ is an absolute weak categorical equivalence. Then $v$ is an absolute weak categorical equivalence if and only if the composite $vu$ is an absolute weak categorical equivalence. \qed
\end{proposition}

We collect some further interesting properties of the class of absolute weak categorical equivalences in the following two propositions.

\begin{proposition} \label{aceprop}
\hspace{1mm}
\begin{enumerate}[font=\normalfont]
\item A morphism of simplicial sets $u \colon A \lra B$ is an absolute weak categorical equivalence if and only if the morphism $u \colon (A,u) \lra (B,1_B)$ is a weak equivalence in the parametrised Joyal model structure on $\sSet/B$.
\item An inner fibration is a trivial fibration if and only if it an absolute weak categorical equivalence.
\item A morphism of simplicial sets is an absolute weak categorical equivalence if and only if it factors as an inner anodyne extension followed by a trivial fibration.
\item A monomorphism of simplicial sets is inner anodyne if and only if it is an absolute weak categorical equivalence.
\end{enumerate}
\end{proposition}
\begin{proof}
(1) Necessity is immediate from the definition, while sufficiency follows from Proposition \ref{quil}.

(2) Let $p \colon X \lra B$ be an inner fibration. Then $p \colon (X,p) \lra (B,1_B)$ is a fibration between fibrant objects in the parametrised Joyal model structure on $\sSet/B$, and is therefore a weak equivalence therein if and only if it is a trivial fibration. Hence the result follows from part (1).

(3) Sufficiency follows from Example \ref{aceex} and Proposition \ref{twothirds}. To prove necessity, let $u$ be an absolute weak categorical equivalence, and let $u = pi$ be a factorisation of $u$ as an inner anodyne extension $i$ followed by an inner fibration $p$. It then follows from Example \ref{aceex} and Proposition \ref{twothirds} that $p$ is an absolute weak categorical equivalence, and hence a trivial fibration by part (2).

(4) Necessity follows from Example \ref{aceex}, while sufficiency follows from part (3) by the retract argument.
\end{proof}

\begin{observation}
Let $\tau_0 \colon \sSet \lra \Set$ denote the functor that sends a simplicial set to the set of isomorphism classes of objects of its homotopy category. It follows from Proposition \ref{enrichie} and part (1) of Proposition \ref{aceprop} that a morphism of simplicial sets $u \colon A \lra B$ is an absolute weak categorical equivalence if and only if the function
\begin{equation*}
\tau_0\Fun_B(u,(X,p)) \colon \tau_0\Fun_B((B,1_B),(X,p)) \lra \tau_0\Fun_B((A,u),(X,p))
\end{equation*}
is a bijection for every inner fibration $p \colon X \lra B$.
\end{observation}

\begin{proposition} \hspace{1mm}
\begin{enumerate}[font=\normalfont]
\item Every absolute weak categorical equivalence is a surjective-on-$0$-simplices weak categorical equivalence. 
\item Any surjective-on-$0$-simplices weak categorical equivalence whose codomain is a quasi-category is an absolute weak categorical equivalence.
\item Not every surjective-on-$0$-simplices weak categorical equivalence is an absolute weak categorical equivalence.
\end{enumerate} 
\end{proposition}
\begin{proof}
(1) This follows from part (3) of Proposition \ref{aceprop}. 

(2) Let $u \colon A \lra B$ be a surjective-on-$0$-simplices weak categorical equivalence and suppose that $B$ is a quasi-category. Let $u = pi$ be a factorisation of $u$ as an inner anodyne extension $i$ followed by an inner fibration $p$. Hence $p$ is an inner fibration between quasi-categories which is surjective-on-$0$-simplices and a  weak categorical equivalence. It follows that the functor $\ho(p)$ is a surjective equivalence, and in particular an isofibration, and hence that $p$ is an isofibration. So $p$ is both an isofibration and a weak categorical equivalence, and is therefore a trivial fibration. Hence $u$ is an absolute weak categorical equivalence by part (3) of Proposition \ref{aceprop}.

(3) By \cite{contrasample}, there exists a morphism of simplicial sets which is a monomorphism, surjective on $0$-simplices, and a weak categorical equivalence, but which is not inner anodyne, and therefore, by part (4) of Proposition \ref{aceprop}, not an absolute weak categorical equivalence. 
\end{proof}

Finally, we recover Stevenson's theorem \cite[Theorem 1.5]{MR3812459} that the class of inner anodyne extensions has the right cancellation property. We used a corollary \cite[Lemma 2.5]{MR3812459} of this theorem in the proof of Proposition \ref{reflprop}.

\begin{theorem}[Stevenson] \label{stevthm}
Let $u \colon A \lra B$ and $v \colon B \lra C$ be a composable pair of monomorphisms of simplicial sets. If $u$ and $vu$ are inner anodyne, then  $v$ is inner anodyne.
\end{theorem}
\begin{proof}
This follows from Proposition \ref{twothirds} and part (4) of Proposition \ref{aceprop}.
\end{proof}

\end{document}